\documentclass[a4paper,twoside,11pt,reqno]{amsart}      
\usepackage[utf8]{inputenc}
\usepackage[T1]{fontenc} %pour certains caractères spéciaux (ex polonais)
\usepackage{amssymb,amsfonts,amsmath,stmaryrd,bbm}
\usepackage{latexsym}
\usepackage{gensymb}
\usepackage{verbatim} 
\usepackage{colortbl}
\usepackage{soul}
\usepackage{arydshln}  % pour les tables ; à placer APRES colortbl ! 
\usepackage{algpseudocode,algorithm}

\usepackage{kpfonts}
\usepackage{afterpage}
\usepackage{dsfont}
\usepackage{color} 
\usepackage{a4wide} %%Smaller margins = more text per page.
\usepackage{url}
\usepackage{pifont, tikz, subfigure}
\usetikzlibrary{plotmarks,shapes,arrows,positioning}

\usepackage[plainpages=false,pdfpagelabels,colorlinks=true,citecolor=blue,hypertexnames=false]{hyperref}

\newcommand{\LandauO}{\mathcal{O}}

\newcommand{\bx}{\bar x}
\newcommand{\bz}{\bar z}

\newcommand{\Opr}{L}
\newcommand{\OpRight}{\Lambda}
\newcommand{\OpLeft}{\Omega}

\newcommand{\Hy}{H} % like hypermaps, here bipartite maps with degrees 1, 2, 3
\newcommand{\B}{B} % idem, with degrees 2, 3

\def\spec{\Theta}

\def \dtwo#1{\frac{\partial}{\partial {#1}_2}}
\def\deriv#1#2{\frac{\partial #1}{\partial #2}}

\def\denom{1-z^2p_2q_2}
\def\denomm{1-\nub\nuw}
\def\den{(1-\nub \nuw)}
\def\spec{\Theta}
\def\semiring{\mathbb A}
\def\nuw{\nu_{\circ}}
\def\nub{\nu_{\bullet}}
\def\Aw{A_{\circ}}
\def\Ab{A_{\bullet}}
\def\bpsi{\psi}%{\overline\Psi}

\def \dtwo#1{\frac{\partial}{\partial {#1}_2}}

\def\change{\Phi}
\def\ichange{\Psi}
\def\Fw{\Upsilon_{\circ}}
\def\Fb{\Upsilon_{\bullet}}

\catcode`\@=11
\def\section{\@startsection{section}{1}%
  \z@{.7\linespacing\@plus\linespacing}{.5\linespacing}%
  {\normalfont\Large\bfseries\scshape\centering}}

\def\subsection{\@startsection{subsection}{2}%
  \z@{.5\linespacing\@plus\linespacing}{.5\linespacing}%
  {\normalfont\large\bfseries\scshape}}

\def\subsubsection{\@startsection{subsubsection}{3}%
  \z@{.5\linespacing\@plus\linespacing}{-.5em}
  {\normalfont\large\bfseries}}
\catcode`\@=12

\newtheorem{theorem}{Theorem}[section]
\newtheorem{lemma}[theorem]{Lemma}
\newtheorem{conjecture}[theorem]{Conjecture}
\newtheorem{cor}[theorem]{Corollary}
\newtheorem{prop}[theorem]{Proposition}

\theoremstyle{definition}

\newtheorem{remark}[theorem]{Remark}

\newcommand{\beq}{\begin{equation}}
  \newcommand{\eeq}{\end{equation}}

\renewcommand{\epsilon}{\varepsilon}
\newcommand{\m}{\mathfrak m}

\newcommand{\ns}{\mathbb{N}}%{\mbox{\bbold N}}

\newcommand{\qs}{\mathbb{Q}}%{\mbox{\bbold Q}}

\newcommand{\rs}{\mathbb{R}}%{\mbox{\bbold R}}
\newcommand{\GK}{\mathbb{K}}

\newcommand{\Sn}{\mathfrak S}

\DeclareMathOperator{\Pol}{Pol}

\DeclareMathOperator{\vv}{v}
\DeclareMathOperator{\ff}{f}
\DeclareMathOperator{\gen}{g}
\DeclareMathOperator{\ee}{e}
\DeclareMathOperator{\he}{h}

\newcommand{\eps}{\epsilon}

\newcommand{\fps}{formal power series}
\newcommand{\gf}{generating function}
\newcommand{\gfs}{generating functions}
\def\emm#1,{{\em #1}}

\newcommand{\Maple}{{\sc Maple}}
\begin{document}
\date{\today}
% --------------------------------------
% 
% --------------------------------------
\title{The Ising model on cubic maps: arbitrary genus}

\author[M. Bousquet-M\'elou]{Mireille Bousquet-M\'elou}

\thanks{MBM was partially supported by the ANR projects DeRerumNatura (ANR-19-CE40-0018), Combiné (ANR-19-CE48-0011), and CartesEtPlus (ANR-23-CE48-0018). AC is fully supported by the Austrian Science Fund (FWF) 10.55776/F1002. BL was partially supported by CartesEtPlus (ANR-23-CE48-0018).}

\address{MBM: CNRS, LaBRI, Universit\'e de Bordeaux, 351 cours de la
  Lib\'eration,  F-33405 Talence Cedex, France} 
\email{bousquet@labri.fr}

\author[A. Carrance]{Ariane Carrance}
\address{AC: Fakultät für Mathematik, Universität Wien, Oskar-Morgenstern-Platz 1, 1090 Vienna, Austria} 
\email{ariane.carrance@univie.ac.at}

\author[B. Louf]{Baptiste Louf}
\address{BL: CNRS, IMB, Universit\'e de Bordeaux, 351 cours de la
  Lib\'eration,  F-33405 Talence Cedex, France} 
\email{baptiste.louf@math.u-bordeaux.fr}


\makeatletter
\@namedef{subjclassname@2020}{%
  \textup{2020} Mathematics Subject Classification}
\makeatother


\begin{abstract} 
  We design a recursive algorithm to compute the partition function of the Ising model, summed over all cubic maps with fixed size and genus. This partition function counts vertex colorings  in black and white, with a weight $\nub$ (resp. $\nuw$) per edge having two black (resp. white) endpoints.
  The algorithm runs in polynomial time, which is much faster than methods based on a Tutte-like, or topological, recursion.

  We construct this algorithm out of a partial differential equation that we derive from the first equation of the KP hierarchy satisfied by the generating function of bipartite maps. This series is indeed related to the Ising partition function by a classical change of variables. We also obtain inequalities on the coefficients of this partition function. This  should be useful for a probabilistic study of cubic Ising maps whose genus grows linearly with their size, in analogy to what was done recently for cubic maps with no additional model.
\end{abstract}

\maketitle

%%%%%%%%%%%%%%%%%%%%%%%%%%%%%%%%%%%%%%%%%%%%%%%%%%%%%%%%%%%%%%%%%%%%%%%% 
\section{Introduction}
%%%%%%%%%%%%%%%%%%%%%%%%%%%%%%%%%%%%%%%%%%%%%%%%%%%%%%%%%%%%%% 

\emph{Combinatorial maps} are discrete surfaces formed by gluing polygons along their sides, or alternatively, graphs embedded on surfaces. Their study goes back to the 60's, where Tutte enumerated {families} of \emph{planar maps} (\emm i.e., maps drawn on the sphere) in a series of papers, including~\cite{tutte-census-maps,tutte-triangulations,tutte-general}. Later on, his approach was extended to maps {drawn} on a surface of arbitrary fixed genus, {yielding} algorithms that compute the numbers of maps recursively in the  genus and the \emm size,  (the number of edges), and 
asymptotic enumerative results as the size  grows large, but the  genus is fixed~\cite{walsh-lehman-I,Bender-Canfield-orientable}. These algorithms were later shown to fit in a larger framework of enumerative geometry, called 
\emph{topological recursion} (see~\cite{eynard-book} and references therein). However,  algorithms based on topological recursion have a superexponential runtime in the genus, and they cannot be used in practice to compute numbers of maps beyond a very small genus.

Fortunately, there are faster ways of computing these numbers: these alternative methods rely on the fact that the generating function of maps satisfies the \emph{KP hierarchy} (KP for Kadomtsev/Petviashvili), an integrable family of nonlinear PDEs, which first arose in
mathematical physics; see~\cite{MiwaJimboDate} and references therein. The KP hierarchy is an extension of the more classical \emph{KdV hierarchy} (Korteweg/de Vries), which models waves in shallow water, but  also occurs in enumerative geometry~\cite{Witten,Kontsevich}. Goulden and Jackson proved that the generating function of bipartite maps satisfies the KP hierarchy~\cite{goulden-jackson-KP}, and derived from this  
a simple recursion for the number of \emm triangulations, (maps whose faces are triangles) {of fixed size and} genus, see~\eqref{GJ-rec}. This recursion yields at once a polynomial time algorithm that computes the associated numbers. It then turned out that this recursion had already {appeared} in the physics literature{, see for instance}~\cite{KKN}.

Around the same time, recursions for another model of enumerative geometry, \emph{Hurwitz numbers}, were obtained thanks to the \emph{Toda hierarchy} (an extension of the KP hierarchy), see~\cite{Okounkov,DYZ}.

Later on, several other fast recursions for counting combinatorial maps were derived thanks to integrable hierarchies~\cite{CarrellChapuy,Kazarian-Zograf,LoufFormula,BCD-non-oriented}.  
Let us mention that these compact recursions still resist any direct combinatorial interpretations, except in a few specific cases~\cite{chapuy-structure-unicellular,CFF13,Louf-bijection,schabanel-bij}.

In the language of physics, maps
provide a model of two-dimensional quantum gravity ``without matter''. But it is natural to consider {more generally} maps endowed with a model of statistical mechanics (see~\cite{DFGZJ} for a detailed review). In this paper, we consider the particular case of the \emph{Ising model}, which consists in bicoloring the vertices  of our maps,
and introducing a weight on \emm monochromatic, edges, that is, edges with both endpoints of the same color. The \emm partition function, then counts bicolored maps of fixed size and genus, with  weighted monochromatic edges.
In a physics setting, the two possible colors are seen as $\pm1$ spins. The Ising model was first introduced a century ago on fixed lattices~\cite{Lenz,Ising}, and has been an extremely active field of study ever since. The Ising model on random maps was introduced in the eighties
as a model of ``2D quantum gravity with matter''~\cite{Bershadsky}.
It has been solved exactly on several families of planar maps~\cite{Ka86,BK87,mbm-schaeffer-ising,eynard2matrix}, 
and shown to satisfy the topological recursion, see \emm e.g.,~\cite{eynard-book}.  However, once again the {underlying algorithm that computes the Ising partition function for maps of fixed size and genus is}
exponential in the genus. Our main result is to give a polynomial 
algorithm that computes these partition functions. We restrict our attention to \emm cubic, maps, that is, maps in which all vertices have degree $3$ (Figure~\ref{fig:ising-ex}). A simple duality transforms these Ising cubic maps in triangulations with bicolored \emm faces,.

\begin{figure}[htb]
  \centering
\includegraphics[scale=0.9]{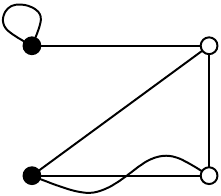}
  \caption{An Ising cubic map of genus $1$ having $2n=4$ vertices, $2$ faces, $3n=6$ edges, and $1$ monochromatic black (resp. white) edge. Its half-edges can be labeled in  $(6n)!=12!$ ways, and a root corner can be chosen in $6n=12$ ways.}
  \label{fig:ising-ex}   
\end{figure}

Our first main result deals with \emm labeled, cubic maps (say, with $3n$ edges), where each half-edge  carries a label between $1$ and $6n$, all labels being distinct.

\begin{theorem}
\label{thm PDE intro}
Let $I_{n,g}(\nub,\nuw)$ be the polynomial that counts labeled cubic maps with~$2n$ vertices {(and thus $3n$ edges)} colored either black or white, with a weight $\nuw$ (resp. $\nub$) per white-white (resp. black-black) edge. 

The series $I:=\sum_{n,g\geq 0}t^{3n}s^g I_{n,g}(\nub,\nuw){/(6n)!}$ satisfies an explicit polynomial PDE in the variables $t$, $\nuw$ and $\nub$, whose coefficients are polynomials in $t$, $\nuw$, $\nub$ and $s$. It has degree $2$ and order~$4$, see Theorem~\ref{thm:PDE-Ising} for details. This equation yields an algorithm to compute the polynomials $I_{n,g}$ in polynomial time and space with respect to $n$ and $g$.  
\end{theorem}

With $3$ edges for instance, that is, $n=1$, the genus can be $0$ or $1$, and one has
  \[
    I_{1,0}
    %= 2\cdot 5!+  4\cdot 5!\nuw^3+ 4\cdot 5!\nub^3+ 6! \nub\nuw,
    = 5! ( 2+  4\nuw^3+ 4\nub^3+ 6 \nub\nuw), \qquad I_{1,1}
    % =2 \cdot 5! + 5!\nuw^3+ 5!\nub^3,
    =5!(2 + \nuw^3+ \nub^3),
  \]
  as illustrated  in Figure~\ref{fig:small-ising}.

Note that since we have two different Ising weights $\nuw$ and $\nub$, this is equivalent to studying the Ising model coupled with a magnetic field. The Ising model without magnetic field corresponds to setting $\nuw=\nub$.

A corollary of our main theorem is the following enumerative bound, dealing now with cubic maps that carry no label but have a distinguished root corner (see Section~\ref{sec:def} for precise definitions).

\begin{cor}\label{thm_bounds}
  Let $\vec I_{n,g}(\nub,\nuw)=  I_{n,g}(\nub,\nuw)/(6n-1)!$ be the the polynomial that counts rooted unlabeled cubic maps with~$2n$ vertices  colored either black or white, with a weight $\nuw$ (resp. $\nub$) per white-white (resp. black-black) edge. For $\nub,\nuw>0$ and $n\geq 5$ the following inequality holds: 
\beq
  \label{ineq}
n\vec I_{n,g}(\nub,\nuw)\geq C(\nub,\nuw)\left(n^3\vec I_{n-2,g-1}(\nub,\nuw)+\sum_{i+j=n-2\atop h+k=g}i\vec I_{i,h}(\nub,\nuw)j\vec I_{j,k}(\nub,\nuw)\right)
\eeq
where
\[
  C(\nub,\nuw)=\frac{\min(\nuw^2,\nub)^4}{\nuw^2+\nub+\frac{\nuw}{\nub}+\frac{1}{\nuw}}.
\]
\end{cor}

This bound can be seen as a weak Ising counterpart of the Goulden--Jackson recurrence relation~\cite{goulden-jackson-KP} giving the number $\vec M_{n,g}$ of uncolored rooted cubic maps with $2n$ vertices and genus $g$:
   \beq\label{GJ-rec}
    (n+1)  \vec M_{n,g}=4 n (3n-2)(3n-4) \vec M_{n-2,g-1} +4 \sum_{i+j=n-2 \atop h+k=g} (3i+2)(3 j+2) \vec M_{i,h}\vec M_{j,k},
  \eeq
  with initial condition $\vec M_{n,g}= \delta_{n,0} \delta_{g,0} -\frac 1 2 \delta_{n,-1} \delta_{g,0}$ for $n\le 0$ or $g<0$.
A bound on $\vec M_{n,g}$, derived from this recursion and analogous to~\eqref{ineq}, has been a crucial ingredient in a work of Budzinski and the third author that establishes the local convergence of random high genus triangulations~\cite{BudzLouf}. We expect that the above bound will play the same role in the study of high genus Ising triangulations (or cubic maps).

\medskip
Our approach starts from the Goulden--Jackson result, and more precisely from the first equation of the KP hierarchy. We first derive from it a PDE for the series counting bipartite maps with vertices of degree $2$ and $3$ only. We then
note that the latter series coincides with the Ising generating function of cubic maps up to a change of variables. It is worth mentioning that all known enumerative results on the Ising model on maps rely on this correspondence. 
We then convert the PDE obtained for bipartite maps into  the PDE of Theorem~\ref{thm:PDE-Ising}. This PDE allows us to compute  the polynomials $I_{n,g}$ of Theorem~\ref{thm PDE intro} recursively. The underlying recurrence relation involves not only the size and genus, as in  Goulden and Jackson's recursion~\eqref{GJ-rec}, but also the number of black and white monochromatic edges. This is to be expected given the higher complexity of our model.

The structure of the paper is as follows. In Section~\ref{sec:bip}, we recall the first equation of the KP hierarchy satisfied by the \gf \ of labeled bipartite maps, and derive from it
a PDE for bipartite maps having only vertices of degree 2 and 3. In Section~\ref{sec:Ising}, we first explain the correspondence between these bipartite maps and cubic maps endowed with the Ising model, and then derive from it  a  PDE for the Ising series. In Section~\ref{sec:unique}, we show that this PDE, along with minimal assumptions, characterizes the \gf \ of Ising cubic maps. The underlying recurrence relation yields the desired polynomial-time algorithm to compute its coefficients.
In Section~\ref{sec:special cases}, we consider the specialization of this PDE to three cases: planar cubic maps,  uncolored cubic maps (where we recover the above Goulden--Jackson recursion) and finally unicellular Ising cubic maps (maps with a unique face). In Section~\ref{sec:inequalities} we prove the bound of Corollary~\ref{thm_bounds}. Finally, in Section~\ref{sec:rooted}, we derive a PDE on the \gf \ of \emph{rooted} unlabeled Ising cubic maps, and check it in the planar case, where a rational parametrization of the series is known.
A \Maple\ session and a {\sc SageMath} file, available on our web pages, accompany this paper.

%%%%%%%%%%%%%%%%%%%%%%%%%%%%%%%%%%%%%%%%%%%%%%%%%%%%%%%%%%%%%%%%%%%%%%%% 
\section{Bipartite maps with vertex degrees 2 and 3}
\label{sec:bip}
%%%%%%%%%%%%%%%%%%%%%%%%%%%%%%%%%%%%%%%%%%%%%%%%%%%%%%%%%%%%%% 

% ================================
\subsection{Definitions}
\label{sec:def}
% ================================

A map is a {\em cellular embedding} of a connected graph (with loops and multiple edges allowed) into a 2-dimensional compact {orientable} surface
without boundaries; here, {\em embedding} means that vertices are distinct and edges do not cross
except at vertices, while {\em cellular} means that the complement of the graph in the surface is homeomorphic to a disjoint union of open disks, called the \emm faces, of the map.
Maps are considered up to orientation-preserving homeomorphisms. Equivalently, a map is defined from its underlying graph by cyclically ordering the half-edges around the vertices. We can thus define the {\em corners} incident to a vertex, as the angular sectors between consecutive edges incident to this vertex. The \emm genus, $\gen(\m)$ of a map $\m$ is the genus of the underlying surface. It can be computed by Euler's formula $\vv(\m)+\ff(\m) -\ee(\m) =2-2\gen(\m)$, where the functions $\vv$, $\ee$ and $\ff$ give  the number of vertices, edges and faces of the map, respectively. A {\em rooted} map is a map with a distinguished corner, called \emm root corner,.
The \emph{dual} of a map $\m$, denoted $\m^*$, is the map obtained by placing a 
vertex of $\m^*$ in each face of $\m$ and an edge of $\m^*$ across each
edge of $\m$. A map is \emm cubic, if all its vertices have degree $3$. Its dual is then  a \emm triangulation,, meaning that all its faces have degree~$3$.

\medskip

We consider in this section  edge-labeled bipartite maps  of arbitrary genus, in which all vertices have degree $1$, $2$ or $3$. We call such maps, for short, \emm maps of bounded degree,. We will count families of such maps by edges (variable $z$), faces (variable $u$), white vertices of degree $i \in \{1, 2, 3\}$ (variable $p_i$), and black vertices of degree~$i$ (variable $q_i$).
We sometimes call these variables the \emm weights, of the corresponding items (edges, faces, etc.).  Our \gfs\ are exponential in the number of edges. In particular, we denote by $\Hy$ the \gf\ of bipartite maps (also called \emm hypermaps,, hence the notation) of bounded degree, that is,
\beq\label{H-def}
\Hy= \sum_{\m} \frac{z^{\ee(\m)} }{\ee(\m)!} u^{\ff(\m)} \prod_{i=1}^3 p_i^{\vv_i^\circ(\m)}  q_i ^{\vv_i^\bullet(\m)},
\eeq
where the function $\vv_i^\circ$ (resp.  $\vv_i^\bullet$) counts white  (resp. black) vertices of degree~$i$. The sum runs over all labeled bipartite maps $\m$ satisfying the above bounded degree condition. We call \emm leaves, the vertices of degree $1$. Note that the variable~$z$ is redundant, since
\[
  \ee(\m)=\sum_{i} i  \vv_i^\circ(\m) = \sum_i i \vv_i^\bullet(\m).
\]
This series starts as follows:
\begin{multline*}
  \Hy=    
u  p_1  q_1 z 
  +u\left( p_1^{2} q_2  + p_2 \,q_1^{2}  +p_2 q_2 \,u\right) \frac{z^{2}}2
  \\ +u \left(p_1^{3} q_3  + p_3 \,q_1^{3}  +3 p_1 p_2 q_1 q_2 + 3 p_1 p_2 q_3 \,u+3p_3 q_1 q_2 \,u+ p_3 q_3 \,u^{2}+ p_3   q_3  \right) \frac{z^{3}}3
  +\LandauO\left(z^{4}\right).
\end{multline*}
See Figure~\ref{fig:small-maps} for an illustration of the coefficient of $z^3/3!$. These labeled maps, say with $n$ edges, can be encoded by three permutations in the symmetric group $\Sn_n$, denoted $\sigma_\circ$, $\sigma_\bullet$ and $\phi$, that describe the cyclic orders of edge labels around white vertices, black vertices, and faces, respectively. The product $\sigma_\circ \sigma_\bullet\phi$ is the identity, and the group generated by these three permutations acts transitively on $\{1, \ldots, n\}$. This encoding is central in the proof that bipartite maps satisfy the KP hierarchy~\cite{goulden-jackson-KP}.

\begin{figure}[htb]
  \centering
  \scalebox{0.9}{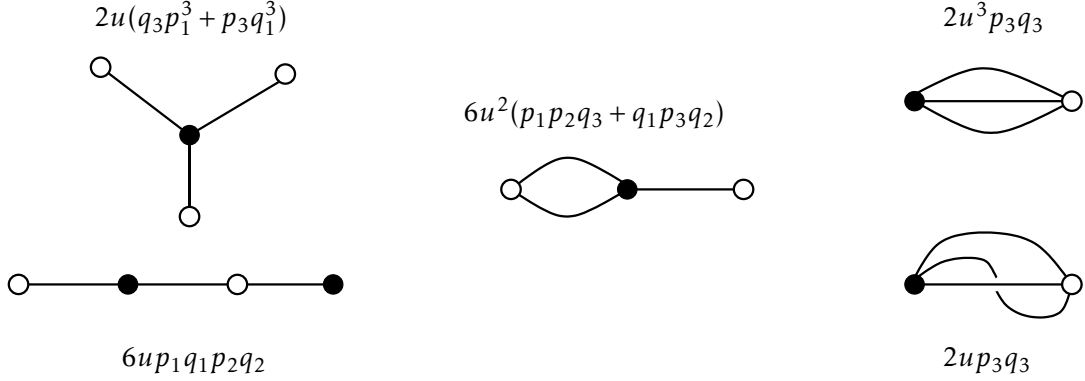}
  \caption{Bipartite maps with $3$ edges and vertex degrees at most $3$. The weights keep track of the number of labelings of the edges and of the exchange of colors.}    
  \label{fig:small-maps}   
\end{figure} 

Given a field~$\GK$ and a variable $x$, the ring of polynomials in $x$ with coefficients in~$\GK$  is denoted $\GK[x]$, the field of rational functions $\GK(x)$, and the ring of \fps\ $\GK[[x]]$. This notation is extended to fractions and series in several variables. For instance, the above series $H$ belongs to $\qs[u, p_1, p_2, p_3, q_1, q_2, q_3][[z]]$.

% =============================================
\subsection{The first KP equation}
% =============================================

In 2008, Goulden and Jackson proved that the series $\Hy$ defined by~\eqref{H-def} satisfies the partial differential equations of the KP hierarchy~\cite[Thm.~3.1]{goulden-jackson-KP}. This is actually true as well without the bound on
the degrees. Here we shall use the first of these partial differential equations (PDE) only. It involves derivatives with respect to the three variables $p_i$.

\begin{theorem}\label{thm:KP}
  The above series $\Hy$ satisfies the following fourth order partial differential equation:
  \beq\label{eq_KP}
    \Hy_{1,3}=\Hy_{2,2}+\frac{1}{12}\Hy_{1,1,1,1}+\frac{1}{2}(\Hy_{1,1})^2,
  \eeq
  where an index $i$ indicates a partial derivative in the variable $p_i$.
\end{theorem}

Our objective in this section is to derive a PDE for the  specialization of $\Hy$ at $p_1=q_1=0$. That is, we want an equation for bipartite maps having vertices of degree $2$ and $3$ only. For such a map $\m$, we have
\beq\label{v-e}
\ee(\m)= 2 \vv_2^\circ(\m)+ 3 \vv_3^\circ(\m) = 2 \vv_2^\bullet(\m)+ 3 \vv_3^\bullet(\m) ,
\eeq
hence the variables $p_3$ and $q_3$ are redundant if we keep $z$, $p_2$ and $q_2$. Let $\spec$ be the operator that sets the variables $p_1$ and $q_1$ to $0$, $p_3$ and $q_3$ to $1$, and let  $\B:=\spec H$. That is,
\beq\label{B-def}
B= \sum_{\m} \frac{z^{\ee(\m)} }{\ee(\m)!} u^{\ff(\m)} p_2^{\vv_2^\circ(\m)}  q_2 ^{\vv_2^\bullet(\m)},
\eeq
where the sum runs over edge labeled bipartite maps with degrees $2$ and~$3$. 
The set of such maps will be denoted by $\mathcal B$.

The following proposition, based on combinatorial constructions, will allow  us to specialize the PDE~\eqref{eq_KP} at $p_1=q_1=0$, $p_3=q_3=1$ (see Corollary~\ref{cor:KP23}).

\begin{prop}\label{prop:spec}
  Let us introduce the following linear differential operator: 
  \beq\label{op:def}
    \Opr=\frac 2{\denom}\left({z^2q_2}\dtwo{p}+{z} \dtwo{q}\right).
  \eeq
  Then the partial derivatives occurring in Theorem~\ref{thm:KP}, once specialized at $p_1=q_1=0$, $p_3=q_3=1$,
  are given by
  \beq\label{H11-spec}
    \spec H_{1,1}=\Opr^2 B+\frac{uz^2q_2}{\denom},
  \eeq
  \beq\label{H1111-spec}
       \spec H_{1,1,1,1}
    =  \Opr^4 B +\frac{12 z^{6} u \left(q_2^{5}  z^{4}+2 q_2^{2}z+p_2 \right)}{(\denom)^{5}},
  \eeq
  \beq\label{H13-spec}
  \spec H_{1,3}
      =\frac 1 {3}\left(z \frac{\partial}{\partial z}-2p_2\deriv{}{p_2}-1\right)\Opr B,
  \eeq
  and finally
  \[ 
    \spec H_{2,2}=\frac{\partial^2 B}{\partial p_2^2},
  \] 
  where the series $B$ is defined by~\eqref{B-def}.
\end{prop}

\begin{proof} The enumeration of labeled  objects, and the correspondence between operations on these objects and operations on their \gfs,  is conveniently described  using the notion of \emm species,. We refer to the book by Bergeron, Labelle and Leroux~\cite{BLL-book}, or to the short, self-contained account of~\cite[Sec.~1]{PiSaSo}. The two ingredients that we need here are the  product of two species, in which the labels are distributed over the two objects, and the marking of an unlabeled element, here a vertex.

  \begin{figure}[htb]
    \centering
    \scalebox{0.7}{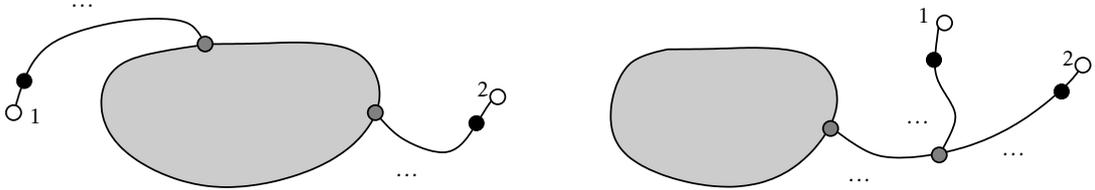}
    \caption{Bipartite maps with two ordered  white leaves and at least one vertex of degree $3$.}
    \label{fig:2-leaves} 
  \end{figure}    
  
  Let us begin with $\spec \Hy_{1,1}$. This series counts maps of bounded degree having exactly two (ordered, unweighted) leaves, both white. There are two types of such maps: with, and without vertices of degree $3$.
  \begin{itemize}
  \item Maps with  no vertex of degree $3$. They are reduced to a chain of alternatingly black and white vertices, with  a white leaf at both ends, oriented from the first to the second leaf. In particular, the chain contains an even number, say $2n$, of edges. A chain of length $2n$ is then simply encoded by a permutation of $\Sn_{2n}$ (giving the sequence of edge labels from end to end), and the contribution of maps of this type is
    \[
      u   \sum_{n\ge 1} \frac{z^{2n}}{(2n)!} (2n)! p_2^{n-1} q_2^n=  \frac{uz^2q_2}{1-z^2p_2q_2}.
    \]
  \item Maps with at least one vertex of degree $3$: in these maps, each of the two leaves lies at the end of a chain of vertices of degree $2$, attached at a vertex of degree $3$ (Figure~\ref{fig:2-leaves}).
    % In this case,
    Let us first delete the chain carrying the second leaf: we obtain a new map with bounded degrees and one white leaf. We then repeat this operation with the
   remaining leaf so as to end with a map of degrees $2$ and $3$. Conversely, given a set $\mathcal M$ of labeled bipartite maps with bounded degrees, let $\overrightarrow{\mathcal M}$ be the set of maps obtained as follows: take a map  in $\mathcal M$, choose a vertex $v$ of degree $2$ and one of the two corners at this vertex, and  attach at this corner a chain ending with a marked white leaf; note that the  vertex $v$ has  now degree~$3$. If $M$ denotes  the \gf\ of $\mathcal M$ (with variables $z, u, p_i, q_i$ as in~\eqref{H-def}), the \gf\  of $\overrightarrow{\mathcal M}$ is
    \[
      \overrightarrow M:=   2   p_1 
      \frac{z^2q_2}\denom\frac{\partial M}{\partial p_2}
      +   2   p_1 
      \frac{z}\denom\frac{\partial M}{\partial q_2} = p_1 \Opr M,
    \]
    where   $\Opr$ is defined by~\eqref{op:def}.    The first (resp. second) half of the expression  counts maps such that the chain is attached to a white (resp. black) vertex. In this case, the length of the chain, that is, the number of edges,  has to be    even (resp. odd).

    Now the maps that we need to count are precisely those obtained by first adding to a map of $\mathcal B$  a  chain, ending at a first white leaf, and then a second chain to the resulting map, ending at a second white leaf. The above argument gives the resulting \gf\ as $\Opr^2 B$. This concludes the proof of~\eqref{H11-spec}.
  \end{itemize}

  \medskip
  
  The series $\spec \Hy_{1,1,1,1}$ counts maps {of bounded degree}, having exactly four (ordered, unweighted) leaves, all of them being white. These maps are obtained by adding consecutively two chains ending at marked white leaves to maps counted by $\spec \Hy_{1,1 }$. By the above construction,  $\spec \Hy_{1,1,1,1}= \Opr ^2\spec \Hy_{1,1 }$, and thanks to~\eqref{H11-spec}, this yields the announced result~\eqref{H1111-spec}.
  
  \medskip

  The series $\spec \Hy_{1,3}$ counts maps with {bounded degree} having exactly one (unweighted) leaf, which is white, and in addition a marked white unweighted vertex of degree $3$. These maps are obtained from maps of $\mathcal B$ by first adding a chain ending at a white leaf --- this gives the series $\Opr B$ --- and then marking a white vertex of degree $3$. The resulting \gf\ is
  \[
    \frac\partial {\partial p_3} \Opr B.
  \]
  The maps $\m$ counted by $\Opr B$ satisfy
  \[
    \ee(\m)= 1+ 2 \vv_2^\circ(\m)+ 3 \vv_3^\circ(\m).
  \]
  Hence
  \[
    \frac\partial {\partial p_3} \Opr B =    \frac 1 {3p_3} \left( z\frac\partial {\partial z} -  2 p_2\frac\partial {\partial p_2}-1\right) \Opr B,
  \]
 which gives~\eqref{H13-spec}.

  The final identity is obvious, since the specialization operator $\spec$ commutes with the differentiation with respect to $p_2$.
\end{proof}

We can now convert the KP equation~\eqref{eq_KP} into an equation for bipartite maps with vertices of degree $2$ and $3$.

\begin{cor} \label{cor:KP23}
  The \gf\ $B$ of  bipartite maps having only  vertices of degree $2$ and~$3$, defined by~\eqref{B-def}, satisfies the following partial differential equation:
  \beq\label{eq_KP_bip}
    \frac 1 {3}\left(z\frac \partial{\partial z}
      -2p_2\deriv{}{p_2}-1\right)\Opr B=\\
    \frac{\partial^2 B}{\partial p_2^2}
    +\frac{1}{12}\Opr^4 B+\frac{1}{2}\left(\Opr^2 B\right)^2
    +\frac{uz^2q_2}{\denom}\Opr^2 B+ R,
    \eeq
  where $L$ is defined by~\eqref{op:def} and
\[
    R:= \frac{1}{2}\left(\frac{uz^2q_2}{\denom}\right)^2
    +\frac{ z^{6} u \left(q_2^{5}  z^{4}+2 q_2^{2} z+p_2 \right)}{(\denom)^{5}}.
\]
  This is again a PDE of fourth order in three variables, namely $z$, $p_2$ and $q_2$. 
\end{cor}

\begin{remark}
  Unsurprisingly, the above PDE does not characterize the series $B$. Experimentally, if we prescribe the following form for $B$:
  \[
    B= \sum_{n,i,j} B_{n,i,j}  z^n p_2^i q_2^j ,
  \]
  where the $B_{n,i,j}$ are polynomials in $u$ and
  the sum  is restricted to tuples $(n,i,j) \in \ns^3$ subject to the conditions derived from~\eqref{v-e}:
  \[
n\ge 2, \quad 0 \le i \le n/2,  \quad  0 \le j \le n/2,    \quad   i \equiv -n \!\!   \mod 3,  \quad \text{and} \quad j \equiv -n \! \!  \mod 3,
\]
and then plug this expression in~\eqref{eq_KP_bip}, extracting the coofficient of $  z^n p_2^i q_2^j $ for increasing values of $n$ gives \emm some, relations between the polynomials  $B_{n,i,j}$, but not sufficiently many. However, if we prescribe:
\begin{itemize}
   \item  symmetry  ($B_{n,i,j}= B_{n,j,i}$),
   \item  the values $B_{n,0,0}$  of the polynomials that count (by faces) bipartite \emm cubic, maps with $n$ edges,
  \end{itemize}
  then the solution of the PDE appears to be unique; see our {\sc Maple} session for details. This should be provable along the same lines as  Proposition~\ref{prop:unique} below.
\end{remark}

\begin{remark} 
  The above PDE is not symmetric in $p_2$ and $q_2$, while the series $B$ \emm is, symmetric.
 We can obtain an \emm antisymmetric, PDE for $B$, which in addition does not involve $z$-derivatives, as follows.
  The terms of the PDE~\eqref{eq_KP_bip} that contain $z$-derivatives are $B_{z,p_2}:=\partial ^2B /\partial z \partial p_2$ and  $B_{z,q_2}:=\partial ^2B /\partial z \partial q_2$, and the PDE depends linearly on them. The PDE obtained by exchanging $p_2$ and $q_2$ contains these two terms as well. So we can write both PDEs, solve them for these two derivatives, and finally write $\partial_{q_2}B_{z,p_2}= \partial_{p_2}B_{z,q_2}$. This gives a PDE in $p_2$ and~$q_2$ only, but of order $5$ instead of $4$. Of course, an alternative is to take any (symmetric) combination of the original PDE and of the one obtained by exchanging $p_2$ and~$q_2$ -- but then some $z$-derivatives remain.
\end{remark}

%%%%%%%%%%%%%%%%%%%%%%%%%%%%%%%%%%%%%%%%%%%%%%%%%%%%%%%%%%% 
\section{Ising cubic maps}
\label{sec:Ising}
%%%%%%%%%%%%%%%%%%%%%%%%%%%%%%%%%%%%%%%%%%%%%%%%%%%%%%%%%%%    

In this section we consider the class of  cubic maps,
labeled on \emm half-edges,. We equip them with an Ising model, meaning that their vertices are colored black and white as in Section~\ref{sec:bip}, but now adjacent vertices may share the same color. In this case we say that the edges that join them are \emm monochromatic, (or, for short, white, or black). Observe that for any cubic map~$\m$, we have $2\ee(\m)=3\vv(\m)$, so that the number of edges (resp. vertices) is a multiple of~$3$ (resp. $2$).

We will count these colored cubic maps --- called \emm Ising maps, henceforth --- by the number of edges (variable  $t$), the genus (variable $s$), the number $\ee^\circ$ (resp. $\ee^\bullet$) of monochromatic white (resp. black) edges (variables $\nuw$ and $\nub$, respectively).  Note that a power $t^{3n}$ corresponds to a map with $3n$ edges and $2n$ vertices. Let $I$ be the exponential \gf\ of Ising  maps, labeled on half-edges. Then:
\beq\label{I-def}
I= \sum_\m    \frac{t^{\ee(\m)}}{(2\ee(\m))!}
s^{\gen(\m)} \nuw^{\ee^\circ (\m)}\nub^{\ee^\bullet (\m)}
=  \left(\frac 1 3 (1+s) + \nuw\nub + (\nub^3+ \nuw^3) (\frac 2 3 + \frac s 6) \right)t^3 + \LandauO(t^6).
\eeq
See Figure~\ref{fig:small-ising} for a justification of the coefficient of $t^3/6!$. Observe that the number of bicolored (also called \emm frustrated,) edges is $\ee^{\bullet\circ}(\m)= \ee(\m)-\ee^\bullet(\m) - \ee^\circ(\m)$, and that the numbers of black and white vertices are given by
\[
  3 \vv^\bullet(\m)= \ee^{\bullet\circ}(\m)+ 2 \ee^\bullet(\m), \qquad \text{and} \qquad 3 \vv^\circ(\m)= \ee^{\bullet\circ}(\m)+ 2 \ee^\circ(\m),
\]
so that these numbers are recorded (be it implicitly) in our series. This means that the Ising model that we address includes a \emm magnetic field,. That is, we can record the difference between the number of black and white vertices.  Indeed, we can rewrite  $I$ as
\[
I= \sum_\m    \frac{t^{\ee(\m)}}{(2\ee(\m))!}
s^{\gen(\m)} \nu^{\ee^\circ (\m)+\ee^\bullet (\m)}c^{\vv^\bullet(\m)-\vv^\circ(\m)},
\]
with $\nu=(\nub\nuw)^{1/2}$ and $c=(\nub/\nuw)^{3/4}$. Also, the above identities imply  that $ \ee^\bullet(\m)-\ee^\circ(\m)$ is always a multiple of $3$.

\begin{figure}[htb]  
  \centering
  \scalebox{0.9}{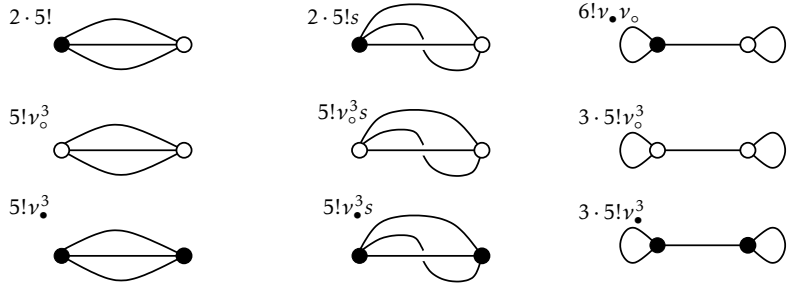}
  \caption{Ising cubic maps with $3$ edges and $2$ vertices. The weights keep track of the number of labelings of half-edges.}
  \label{fig:small-ising} 
\end{figure}  

Starting from an Ising cubic map, one can insert on its edges bicolored chains of vertices of degree $2$ so as to obtain a bipartite map with vertex degrees $2$ and $3$. This allows one to relate the series $I$ to the series $B$ of Section~\ref{sec:bip}. Variants of this  ``trick'' have been used several times in the study of the Ising model on maps~\cite{Ka86,mbm-schaeffer-ising,eynard-book}. We work out its details, in our setting, in Section~\ref{sec:change}. Then we use this in Section~\ref{sec:EDIsing} to convert the PDE satisfied by $B$ into a PDE satisfied by $I$.

% ==============================================================
\subsection{From bipartite maps to Ising maps}
\label{sec:change}
% ==========================================================

\begin{prop}\label{prop:change}
  Define the change of variables $\ichange$ on the ring $\qs[ \nub, \nuw,s][[t]]$ by
  \beq\label{Psi}  
  \begin{split}  
    t\mapsto u^{1/3} \frac{z}{\denom} , 
    & \hskip 10mm \nub\mapsto zp_2,\\
    s\mapsto {u^{-2}},  \hskip 18mm
&\hskip 10mm \nuw\mapsto zq_2.
  \end{split}
  \eeq
  It gives a series of $\qs[  p_2, q_2,{u^{1/3}},1/u][[z]]$. Then the \gf\ $B$ of bipartite maps with vertex degrees $2$ and $3$, defined by~\eqref{B-def}, is
  \beq\label{BI}
    B=  \frac{u^2}{2}\log\left(\frac 1{1-z^2p_2q_2}\right) +u^2 \ichange(I),
  \eeq
  where $I$ is the Ising \gf\ defined by~\eqref{I-def}. Conversely,
  \beq\label{IPhiB}
  I= s \, \change(B) - \frac{1}{2}\log\frac1{1-\nub\nuw},
  \eeq
  where the inverse change of variables $\Phi$ is defined by
  \beq\label{Phi-def}
  \begin{split}
    z\mapsto s^{1/6} t (1-\nub\nuw), & \hskip 10mm 
    p_2\mapsto \frac{\nub}{s^{1/6} t (1-\nub\nuw)},\\
    u\mapsto s^{-1/2},\hskip 15mm  & \hskip 10mm  q_2\mapsto \frac{\nuw}{s^{1/6} t(1-\nub\nuw)}.
  \end{split}
  \eeq
The transformation $\Phi$ maps series in $\qs[ p_2, q_2,u][[z]]$ in which the sum of the exponents of $p_2$ and~$q_2$ never exceeds the exponent of $z$ to series of $\qs[s^{1/6}, s^{-1/6}][[t, \nub, \nuw]]$.
\end{prop}
Note that~\eqref{v-e} implies that the series $B\in \qs[ p_2, q_2,u][[z]]$ satisfies the above condition.
\begin{proof}

 Let us first observe that the series $B$ defined by~\eqref{B-def} is also the exponential \gf\ of bipartite maps (with vertex degrees $2$ and $3$ as before) \emm labeled on  half-edges,, with half-edges weighted by $\sqrt z$. That is,
    \[
      B= \sum_{\m} \frac{z^{\ee(\m)} }{\ee(\m)!} u^{\ff(\m)} p_2^{\vv_2^\circ(\m)}  q_2 ^{\vv_2^\bullet(\m)}
      =\sum_{\m'} \frac{\sqrt z^{\, \he(\m')} }{\he(\m')!} u^{\ff(\m')}  p_2^{\vv_2^\circ(\m')}  q_2 ^{\vv_2^\bullet(\m')},
    \]
    where the function $\he$ counts half-edges, the first sum runs over edge labeled maps $\m$, and the second over half-edge labeled maps $\m'$. The reason for that is that there exists a $(2n)!/n!$-to-1 correspondence between maps $\m'$ with $n$ edges labeled on half-edges and maps $\m$ on~$n$ edges labeled on edges. This correspondence works as follows: starting from $\m'$, we erase all labels that are incident to  white vertices, and  relabel the other half-edges with $1, 2, \ldots, n$, while preserving  their relative order. This gives $\m$. Conversely, starting from $\m$, we first choose in $\{ 1, \ldots, 2n\}$ the $n$ labels of $\m'$ that will be incident to black vertices, spread them on the corresponding half-edges (while preserving the order of labels of $\m$) and then distribute the remaining $n$ labels in any way on the half-edges that are incident to white vertices. We obtain in this way $\binom{2n}{n} n!= (2n)!/n!$ maps $\m'$. This observation will allow us to use the arguments  of the  theory of species~\cite{BLL-book} where now the atoms are half-edges (rather than edges in the proof of Proposition~\ref{prop:spec}).

  \begin{figure}[htb]
    \center
    \includegraphics[scale=0.8]{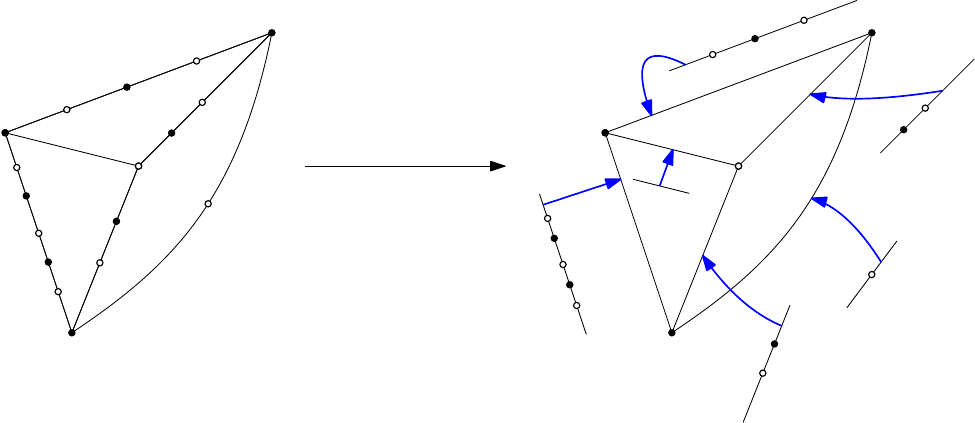}
    \caption{The decomposition of a bipartite map of $\mathcal B$ into a cubic map and  chains. We refer to Figure~\ref{fig:bip-ising} for details on the labels.}
    \label{fig_decompo}
  \end{figure}

 We now embark on the proof of~\eqref{BI}.  The maps $\m'$ (labeled on half-edges) enumerated by~$B$ are of two types.
  \begin{itemize}
  \item Either they only have vertices of degree $2$, in which case they are cycles of even length, say $2n$,  enumerated by
    \[
      2 u^2 \sum_{n\ge 1} \frac{z^{2n}}{(4n)!}(p_2q_2)^n (4n-1)!=
      \frac{u^2}{2}\log\left(\frac 1{1-z^2p_2q_2}\right).
    \]
    This can be seen  by cutting these cycles in the middle of the edge containing the label~$1$, which gives an ordered chain with two dangling half-edges; the factor $2$ accounts for the fact that the label $1$ can be attached to a black or to a white vertex.

  \item Or they have vertices of degree $3$, joined by chains of vertices of degree two  (Figure~\ref{fig_decompo}, left). To such a map $\m'$, we associate an Ising cubic map $\m$ as follows: we erase all vertices of degree $2$, and only retain the labels that are incident to cubic vertices (Figure~\ref{fig:bip-ising}).  If $\m$ (and $\m'$) have $2n$ cubic vertices, so that $\m$ has $6n$ half-edges, we then relabel these half-edges with labels $1, 2, \ldots, 6n$, while preserving their order. Some  chains with dangling half-edges come out, as illustrated in Figure~\ref{fig_decompo}. We orient them as indicated in Figure~\ref{fig:bip-ising}: a chain associated with a monochromatic edge of $\m$ with labels $i<j$ is oriented away from the dangling half-edge that is matched with $i$ in~$\m'$ (such chains contain an odd number of vertices); a non-empty chain with an even number of vertices is oriented from its black to its white endpoint.
  \end{itemize}

  \begin{figure}[htb]
    \centering
    \scalebox{0.8}{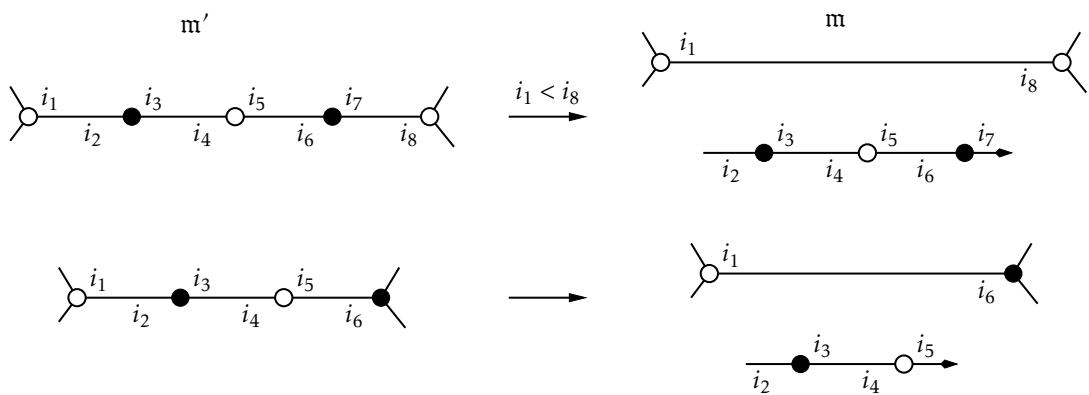}
    \caption{Erasing vertices of degree $2$ in a bipartite map $\m'$ gives an Ising cubic map $\m$. Oriented bicolored chains come out.}
    \label{fig:bip-ising} 
  \end{figure}   

  Let us determine the \gf\ of bipartite maps $\m'$ that yield a given Ising cubic map $\m$ having $2n$ vertices. These maps are obtained from $\m$ by choosing for each edge~$e$ of $\m$  a bipartite oriented chain with dangling half-edges at both ends (Figure~\ref{fig:bip-ising}):
  \begin{itemize}
  \item For a monochromatic white edge $e$ with labels $i<j$, we choose  a chain of $2k+1$ vertices (and $4k+2$ half-edges), with black endpoints. These chains are counted by
    \[
      \sum_{k\geq 0} p_2^kq_2^{k+1} {\sqrt{z}}^{\, 4k+2}= \frac{z q_2}{\denom}.
    \]
   The starting point of the chain will be attached on the side of $e$ labeled $i$. 
  \item Analogously, for a monochromatic black edge $e$, we choose an oriented chain of $2k+1$ vertices, with white endpoints, with \gf $\frac{zp_2}{\denom}$; it will be inserted in $e$ in a canonical fashion again.
  \item Finally, for a bicolored edge $e$, we choose a (possibly empty) chain of $2k$ vertices, starting with a black vertex, with \gf\  $ \frac{1}{\denom}$. This chain will be inserted in $e$ in the only way that preserves bipartiteness.
  \end{itemize}
  Hence, the set of bipartite maps $\m'$ that give $\m$ after erasing vertices of degree $2$ is a labeled product, in the species setting, of $\m$ and of a collection of oriented chains, and has exponential \gf
  \[
    \frac{\sqrt z ^{\, 6 n}}{(6n)!} u^{\ff(\m)} \left(\frac{z q_2}{\denom}\right)^{\ee^\circ (\m)}
    \left(\frac{z p_2}{\denom}\right)^{\ee^\bullet (\m)} \frac{1}{(\denom)^{3n-\ee^\circ (\m)-\ee^\bullet (\m)}}.
  \]
  Recalling that $\m$ has $3n$ edges and $2n$ vertices, and that its genus $\gen(\m)$ satisfies $2\gen(\m)=2+n-\ff(\m)$, this can be rewritten as
  \[
    \frac{u^2}{(2\ee(\m))!} \left( \frac{z u^{1/3}}{\denom}\right)^{\ee(\m)}
    u^{-2\gen(\m)} (z q_2)^{\ee^\circ (\m)}(z p_2)^{\ee^\bullet (\m)}.
  \]
  Comparing with~\eqref{I-def} shows that bipartite maps of $\mathcal B$ that have at least one vertex of degree~$3$ have \gf\ $u^2 \Psi(I)$, where $\Psi$ is the change of variables~\eqref{Psi}.

  The rest of the proof is a mere calculation.
\end{proof}

% ==============================================================
\subsection{A PDE for Ising cubic maps}
\label{sec:EDIsing}
% ==============================================================

Let us return to the PDE~\eqref{eq_KP_bip}, which involves the series $B(z, p_2, q_2,u)$.
% and $u$.
We want to apply to this identity the change of variables $\Phi$ defined by~\eqref{Phi-def}, and to write the resulting identity in terms of the series $I$ (that is, in terms of $\Phi(B)$) and its partial derivatives with respect to $t, \nub, \nuw$ and $s$. We find convenient to use the following notation: 
\[
  \partial _t =\frac{\partial }{\partial t} \qquad \text{and} \qquad    D_t= t\frac{\partial }{\partial t},
\]
  \[
    \partial _\circ =  \frac{\partial }{\partial \nuw}, \qquad \text{and} \qquad    D_\circ= \nuw\frac{\partial }{\partial \nuw},
  \]
  and likewise for $\nub$-derivatives.

\begin{lemma}\label{lem_change}
  Applying the change of variables $\Phi$ to the differential operators involved in~\eqref{eq_KP_bip} yields: 
  \begin{align*}
    \change\circ z\deriv{}{z}&=\left(\frac{1+\nub\nuw}{1-\nub\nuw}\, D_t+ D_\bullet+D_\circ\right) \circ \change, \\
%    \left(\frac{t(1+\nub\nuw)}{1-\nub\nuw}\partial_t+ \nub\partial_\bullet+\nuw\partial_\circ\right) \circ \change, \\
      %       
    \change\circ\deriv{}{p_2}&=s^{1/6} t\left(\nuw D_t+(\denomm)\partial_\bullet\right)\circ \change,\\
%    =s^{1/6} t\left(t\nuw\partial_t+(\denomm)\partial_\bullet\right)\circ \change,\\
      %       
    \change\circ \Opr&=s^{1/3}\OpRight  \circ \change,
  \end{align*}
  where $L$ is defined by~\eqref{op:def} and $\OpRight$ is the following linear operator:
  \beq\label{Lambda-def}
  \OpRight=2t^2\left((\nuw^2+\nub)D_t+ \nuw(1-\nub\nuw)\partial_\bullet+(1-\nub\nuw)\partial_\circ\right).
  \eeq
\end{lemma} 

\begin{proof}
  Let $A\equiv A(z, p_2, q_2, u)$ be a series in $\qs[p_2,q_2,u][[z]]$ such that in all monomials of $A$, the sum of the exponents of $p_2$ and $q_2$ never exceeds the exponent of $z$.  Let $J\equiv J(t,\nub,\nuw,s)=\Phi (A)$. This means conversely  that
  \beq\label{AJ}
    A(z, p_2, q_2, u)= \Psi (J(t,\nub,\nuw,s))= J\left(\bpsi(z, p_2, q_2, u)\right),
  \eeq
  where $\bpsi=\left(\bpsi_1, \ldots, \bpsi_4\right)$ is the following vectorial function (see~\eqref{Psi}):
  \[
    \bpsi: (z, p_2, q_2, u) \mapsto  \left( u^{1/3} \frac{z}{\denom} , 
      zp_2,  zq_2, u^{-2}\right).
  \]
  We differentiate~\eqref{AJ} with respect to $z$ using the chain rule, and then apply $\Phi$:
  \[
    \Phi \left(\frac{ \partial A}{\partial z} (z, p_2, q_2, u)\right)=
    \Phi\left(\frac{\partial \psi_1} {\partial z}\right)  \times \frac{\partial J}{\partial t} (t,\nub,\nuw,s) +\cdots +\Phi\left(\frac{\partial \psi_4}{\partial z}\right) \times \frac{\partial J}{\partial s} (t,\nub,\nuw,s),
  \]
with $J=\Phi(A)$.  So  we need  the image by $\Phi$ of the Jacobian matrix of $\psi$. We compute it to be:
  \[
    \Phi     \begin{pmatrix}
      \frac{\partial \psi_1}{\partial z} & \frac{\partial \psi_1}{\partial p_2}
      &\frac{\partial \psi_1}{\partial q_2} &\frac{\partial \psi_1}{\partial u} \\
      \frac{\partial \psi_2}{\partial z} & \frac{\partial \psi_2}{\partial p_2}
      &\frac{\partial \psi_2}{\partial q_2} &\frac{\partial \psi_2}{\partial u} \\  \frac{\partial \psi_3}{\partial z} & \frac{\partial \psi_3}{\partial p_2}
      &\frac{\partial \psi_3}{\partial q_2} &\frac{\partial \psi_3}{\partial u} \\ \frac{\partial \psi_4}{\partial z} & \frac{\partial \psi_4}{\partial p_2}
      &\frac{\partial \psi_4}{\partial q_2} &\frac{\partial \psi_4}{\partial u} \\        \end{pmatrix}
    =
    \begin{pmatrix}
      \frac{1+\nub \nuw}{s^{1/6}(1-\nub \nuw)^2}&  t^{2}s^{1/6}\nuw &  t^{2}s^{1/6}\nub & \cdots\\
      \frac{\nub}{ts^{1/6}(1-\nub\nuw)}& ts^{1/6}(1-\nub\nuw) & 0 & \cdots\\
      \frac{\nuw}{ts^{1/6}(1-\nub\nuw)}&0 & ts^{1/6}(1-\nub\nuw) &  \cdots\\
      0 &  0 &  0  &  \cdots 
    \end{pmatrix}
    .
  \]
  We ignore the last column  because we never differentiate with respect to $u$ (see Corollary~\ref{cor:KP23}).
  %and $\psi_4$ only depends on $u$.

  We claim that this yields the identities stated in the lemma. Let us examine for instance the first one:
  \begin{align*}
    \Phi \left(z \frac{\partial A}{\partial z}\right) &=
                                                        \Phi\left(z\frac{\partial \psi_1} {\partial z}\right)  \times \frac{\partial J}{\partial t}  +\Phi\left(z\frac{\partial \psi_2}{\partial z}\right) \times \frac{\partial J}{\partial \nub} +\Phi\left(z\frac{\partial \psi_3}{\partial z}\right) \times \frac{\partial J}{\partial \nuw},\\
                                                      &=  t  \frac{1+\nub \nuw}{1-\nub \nuw}\times  \frac{\partial J}{\partial t}  + \nub \times  \frac{\partial J}{\partial \nub } + \nuw \times  \frac{\partial J}{\partial \nuw },
  \end{align*}
  as stated in the lemma.

  The other identities are proved in a similar fashion.
\end{proof}

We can now write a PDE for the Ising \gf.

\begin{theorem}\label{thm:PDE-Ising}
  The \gf\ $I$  of Ising cubic maps, defined by~\eqref{I-def}, satisfies the following fourth order PDE in the variables $t, \nub$ and $\nuw$:
  \beq\label{eq_KP_Ising}
    \OpLeft I=   \frac s {12}\Lambda ^4I
    + \frac 1 2  (\Lambda^2 I)^2
    + t\left( \nuw+2t^3\left(2\nuw^4 + \nub\nuw^2 + 2\nub^2 + 3\nuw\right)\right)\Lambda^2 I
    +t^5 Q,
  \eeq
  where $\Lambda$ is defined by~\eqref{Lambda-def},
  \beq\label{OpLeft-def}
  \OpLeft:=  \frac 1 {3}\left( D_t+D_{\circ}-D_{\bullet}-1\right) \circ \Lambda
  -    \left( t \nuw D_t + t (1-\nub \nuw)\partial _{\bullet}\right)^2,
  \eeq
  and 
  \begin{multline*}
    Q=2\nuw \left(2  \nuw^{4}+  \nub  \nuw^{2}+2  \nub^{2}  +3  \nuw\right)
    +\left( \nuw^{5}+2  \nuw^{2}+ \nub\right) s %\notag
    + 2\left(2\nuw^4 + \nub\nuw^2 + 2\nub^2 + 3\nuw\right)^2t^3
    \\
    +2\left(16  \nuw^{8}+5  \nuw^{6}  \nub+10  \nub^{2}  \nuw^{4}+16  \nub^{3}  \nuw^{2}+59  \nuw^{5}+16  \nub^{4}+54  \nub  \nuw^{3}+37  \nub^{2}  \nuw+32  \nuw^{2}+11  \nub\right) t^3s.
  \end{multline*}
\end{theorem}

We will see in the next section  that this PDE, combined with a degree condition and the fact that $I$ is symmetric in $\nub$ and $\nuw$,  characterizes $I$ in $t\qs[\nub,\nuw,s][[t]]$. Clearly, the PDE itself is not symmetric in $\nub$ and $\nuw$. So far, our efforts to build another PDE that would be both symmetric and smaller have failed.

\begin{proof}
  We apply the change of variables $\Phi$, defined by~\eqref{Phi-def}, to Equation~\eqref{eq_KP_bip}, using the identities of Lemma~\ref{lem_change} and the connection~\eqref{IPhiB} between $\Phi(B)$ and $I$. 

  The left-hand side of~\eqref{eq_KP_bip} gives
  \[
    \frac 1 {3s^{2/3}}\left( D_t+D_{\circ}-D_{\bullet}-1\right) \circ \Lambda I + \frac 1{s^{2/3}} t^2 \nuw^2.
  \]
  The first term on the right-hand side gives
  \beq\label{first}
    \frac 1 {s^{2/3}} \left( t \nuw D_t + t (1-\nub \nuw)\partial _{\bullet}\right)^2 I+
    \frac 1 {2s^{2/3}} t^2  \nuw^2.
  \eeq
  The second one gives
  \begin{multline*}
    \frac {s^{1/3}} {12}\Lambda ^4 I +\\
    2 s^{{1/3}} t^{8} \left(16 \nuw^{8}+5 \nub \,\nuw^{6}+10 \nub^{2} \nuw^{4}+16 \nub^{3} \nuw^{2}+59 \nuw^{5}+16 \nub^{4}+54 \nub \,\nuw^{3}+37 \nub^{2} \nuw +32 \nuw^{2}+11 \nub \right).
  \end{multline*}
  The third one gives
  \[
    \frac 1 {2 s^{2/3}} (\Lambda^2 I)^2
    + \frac 2{s^{2/3}}t^4\left(2\nuw^4 + \nub\nuw^2 + 2\nub^2 + 3\nuw\right)\Lambda^2 I
    + \frac 2{ s^{2/3}}t^8\left(2\nuw^4 + \nub\nuw^2 + 2\nub^2 + 3\nuw\right)^2.
  \]
  The fourth one  gives
  \[
    \frac{  \nuw }{s^{2/3}}t \Lambda^2 I
    + \frac 2 {s^{2/3}} t^5\nuw\left(2\nuw^4 + \nub\nuw^2 + 2\nub^2 + 3\nuw\right).
  \]
  Finally, the fifth and last one  gives
  \[
    \frac 1 {2s^{2/3}} t^2 \left(2\nuw^5st^3 + 4\nuw^2st^3 + 2\nub st^3 + \nuw^2\right).
  \]
  It remains to multiply by $s^{2/3}$ and group all terms not involving $I$ to obtain the announced PDE. Note that the first term of~\eqref{first}
  has been moved to the left-hand side, for reasons that will be explained later.
\end{proof}

%%%%%%%%%%%%%%%%%%%%%%%%%%%%%%%%%%%%%%%%%%%%%%%%%%%%%%%%%%%%%%%%%%% 
\section{Uniqueness and effective calculation of the Ising series}
\label{sec:unique}
%%%%%%%%%%%%%%%%%%%%%%%%%%%%%%%%%%%%%%%%%%%%%%%%%%%%%%%%%%%%%% 
%=====================================
\subsection{Uniqueness}
%=====================================

The first objective of this section is to establish the following result, according to which the PDE that we have obtained for the series $I$, combined with two natural conditions, characterizes this series.

\begin{prop}\label{prop:unique}
  The PDE~\eqref{eq_KP_Ising} satisfied by the Ising series $I$ of cubic maps, defined by~\eqref{I-def}, characterizes $I\equiv I(t,\nub,\nuw,s)$ in the ring of series $J(t,\nub,\nuw,s)$ of $t\qs[\nub,\nuw,s][[t]]$ satisfying the following two conditions:
  \begin{itemize}
  \item for each $n$, the total degree in $\nub$ and $\nuw$ of the coefficient of $t^n$ is bounded by $n$,
  \item $J(t, \nub, 0,s)=J(t,0, \nub,s)$.
  \end{itemize}
  More precisely, the PDE determines the coefficient of $t^n$ in the series $I$  inductively in $n$.
\end{prop}
Observe that we do not need to require that $J$ is a series in $t^3$. The above two conditions are obviously satisfied by $I$, because in the contribution of any Ising cubic map having $n$ edges, the total degree in $\nub$ and $\nuw$ is the number of monochromatic edges, hence bounded by $n$. The symmetry is obvious as well, and more generally $I(t, \nub, \nuw,s)=I(t,\nuw, \nub,s)$.

Before we embark on the proof, let us examine more closely both sides of  our PDE. 
Let $J=\sum_{n\ge 1} t^n J_n$ be a series satisfying the conditions of the proposition, where $J_n$ is a polynomial in $\nub, \nuw$ and $s$.
Then $\Omega (t^n J_n)$ is of the form $t^{n+2} \Omega_n(J_n)$, where $\Omega_n(J_n)$ is independent of~$t$. More precisely, $\OpLeft_n$ is the following linear differential operator:
\begin{multline*}
  \OpLeft_n:=   \frac 2 {3}\left( n+1+D_{\circ}-D_{\bullet}\right) \circ
  \left( (\nuw^2+\nub)n+ \nuw \den \partial _{\bullet} + \den \partial_{\circ}\right)
  \\
  -    \left(  \nuw (n+1) +  (1-\nub \nuw)\partial _{\bullet}\right)\circ\left(  \nuw n+  (1-\nub \nuw)\partial _{\bullet}\right).
\end{multline*}

Moreover, if we replace $I$ by $J$ in the right-hand side of the PDE~\eqref{eq_KP_Ising}, and extract  the coefficient of $t^{n+2}$, then this coefficient only depends of the polynomials $J_k$ up to $k=n-3$ (we moved the first term of~\eqref{first}  to the left-hand side of the PDE to guarantee this property).
Hence, if we know the polynomials $J_0=0, J_1, \ldots, J_{n-3}$, we can try to determine $J_n$ by solving an equation of the form $\OpLeft_n(J_n)= \Pol$, for some explicit polynomial $\Pol$ in $\nub, \nuw$ and~$s$. Since we know that $J=I$ is a solution, it suffices to study the kernel of $\OpLeft_n$.

\begin{lemma}\label{lem:kernel}
  For $n\ge 1$,  the linear operator $\OpLeft_n$, restricted to polynomials $P$ in $\nub$ and $\nuw$ of total degree at most $n$ that satisfy $P(\nub,0)=P(0,\nub)$, has trivial kernel.
\end{lemma}

\begin{proof}
  If $\OpLeft_n(P)=0$, with
  \[
    P(\nub,\nuw)=\sum_{i+j\le n} p_{i,j} \nub^i \nuw^j,
  \]
  then
  \begin{multline*}
    3  \OpLeft_n (P) =
    \sum_{i+j\le n} p_{i,j} \nub^i \nuw^j
    \left( {2 \left(j -n \right) \left(i -j -n \right) \nub} 
      -{\left(i -n \right) \left(i +2 j -n +3\right)  \nuw^{2}}\right. \\
    \left.     -{2 j \left(i -j -n \right)  \nuw^{ -1}}
      +{2 i \left(2 i +j -2 n \right) \nub^{ -1} \nuw}
      -3i \left(i -1\right) \nub^{ -2}
    \right)=0.
  \end{multline*}
  Extracting the coefficient of $\nub^{i+1}\nuw^j$ gives a relation between the coefficients of $P$ that we may try to use to compute the $p_{i,j}$'s
  by decreasing induction on $i$ (the \emm right-to-left, recursion):
  \begin{multline}
    {2 \left(n-j \right) \left(n-i +j  \right) p_{i ,j}}=
    {\left(i +1-n \right) \left(i +2 j -n \right) p_{i +1,j -2}}
    +{2 \left(j +1\right) \left(i -j -n \right) p_{i +1,j +1}}\\
    -{2 \left(i +2\right) \left(2 i +3+j -2 n \right) p_{i +2,j -1}}
    +3\left(i +3\right) \left(i +2\right) p_{i +3,j}. \label{right-left}
  \end{multline}
  However, the left-hand side vanishes at the two extreme points  $(i,j)=(0,n)$ and $(i,j)=(n,0)$. Alternatively,
   replacing in the above recursion $i$ by $i-1$ and $j$ by $j+2$ gives the following \emm top-down, recursion:
  \begin{multline}
    {\left(i -n \right) \left(i +2 j -n +3\right) p_{i ,j}} =
    {2 \left(j +2-n \right) \left(i -3-j -n \right) p_{i -1,j +2}}
    -{2 \left(j +3\right) \left(i -3-j -n \right) p_{i ,j +3}}\\
    +{2 \left(i +1\right) \left(2 i +3+j -2 n \right) p_{i +1,j +1}}
    -3\left(i +2\right) \left(i +1\right) p_{i +2,j +2}.\label{top-down}
  \end{multline}
  Of course, we have the boundary conditions $p_{i,j}=0$ if $i<0$ or $j<0$ or $i+j>n$.

  We first write the top-down recursion~\eqref{top-down} at $i=0,j=n$, and obtain $p_{0,n}=0$. By the symmetry assumption, we also have $p_{n,0}=0$. Now for $i<n$ and $j<n$, the coefficient of $p_{i,j}$ in the right-to-left recursion~\eqref{right-left} does not vanish, and allows us to conclude, by decreasing induction on $i$, that $p_{i,j}=0$ for all $i$, so that the polynomial $P$ is identically zero.
\end{proof}

\begin{remark}\label{rem:compute}
  The procedure used in the above proof also allows us to solve
  the non-homo\-geneous equation $\OpLeft_n(J_n)=\Pol$ where $\Pol$ is the polynomial $\Omega_n(I_n)$, under the assumptions that  $J_n(\nub,0) =J_n(0,\nub)$ and that $J_n$ has total degree at most $n$ in $\nub$ and $\nuw$: we first determine the coefficient of $\nub^0\nuw^n$ using the top-down recursion, then use symmetry to determine the coefficient of $\nub^n \nuw^0$, and then work by decreasing induction on the exponent of $\nub$ using the right-to-left recursion.
\end{remark}

\begin{remark}\label{rem:non-injective}
  If we do not impose  the symmetry condition, the kernel of $\Omega_n$ appears to be trivial when  $n\ge1$  is even, but one-dimensional when $n$ is odd. For instance, one readily checks that
  \[
    \Omega_1(\nub)=0, \qquad \Omega_3(1+\nub ^3)=0, \qquad \Omega_5(\nub^5 + 2\nub^2 + \nuw)=0.
  \]
\end{remark}

\begin{proof}[Proof of Proposition~\ref{prop:unique}]
  Let us consider a series $J$ satisfying the assumptions of the proposition. For $n\ge 0$,  let $J_n$ (resp. $I_n$) denote the coefficient of $t^n$ in $J$ (resp. $I$). Let us prove by induction on $n$ that $J_n=I_n$. By assumption on $J$, this holds for $n=0$. Assume that it holds for $J_0, J_1, \ldots, J_{n-1}$, with $n\ge 1$. Extract from the PDE~\eqref{eq_KP_Ising} the coefficient of $t^{n+2}$. As observed above  Lemma~\ref{lem:kernel}, this gives the equation $\OpLeft_n J_n= \Pol$, where $\Pol$ only involves $\nub, \nuw, s$ and the polynomials $J_1, \ldots, J_{n-3}$ (and their partial derivatives). Since $I$ also satisfies the PDE, and $I_i=J_i$ for $i<n$, we also have  $\OpLeft_n I_n= \Pol$, \emm for the same value of $\Pol$,. We conclude that $J_n=I_n$ thanks to Lemma~\ref{lem:kernel}.
\end{proof}

%===========================================
\subsection{Implementation}
\label{sec:implement}
%===========================================
We have implemented the above recursive calculation of the coefficient $I_n$ of $t^n$ in the series $I$ both in {\sc Maple} and in {\sc SageMath}.
The programs are available on our webpages. We take advantage of the following three properties of $I$:
\begin{itemize}
\item $I_n=0$ unless $n$ is a multiple of $3$,
\item if we write $I_n= \sum_{i,j} I_{n,i,j} \nub^i \nuw^j $,  for $I_{n,i,j}$ a polynomial in $s$, then $I_{n,i,j}=0$ unless $i-j$ is a multiple of $3$, as observed at the beginning of Section~\ref{sec:Ising},
\item  $I_n$ is symmetric in $\nub$ and $\nuw$.
\end{itemize}
With a naive implementation in \Maple\ on a laptop, one can go up to $72$ edges in $40$ seconds, or $120$ edges (maximal genus $20$)  in $15$ minutes. The {\sc SageMath}  implementation is  so far slower due to a  less efficient handling of multivariate polynomials.

%%%%%%%%%%%%%%%%%%%%%%%%%%%%%%%%%%%%%%%%%%%%%% 
\section{Three special cases}
\label{sec:special cases}
%%%%%%%%%%%%%%%%%%%%%%%%%%%%%%%%%%%%%%%%%%%%%% 

The form of the differential operators involved in the PDE satisfied by the Ising series~$I$ allows us to extract at once equations satisfied by three subseries of $I$: those counting planar maps,   monochromatic (white) maps, and unicellular maps. In the second case, we recover, unsurprisingly, the Goulden and Jackson recurrence relation on the number of cubic maps with $3n$ edges and genus $g$.

% ==================================================
\subsection{The planar case}
\label{sec:planar}
% ==================================================
This is the simplest possible specialization of the three. Let $P$ be the Ising \gf\ of cubic planar maps, defined as in~\eqref{I-def} but by restricting the sum to planar maps. It is obtained by setting $s=0$ in $I$.

\begin{cor}\label{cor:planar}
  The Ising \gf\ $P$ of cubic planar maps (labeled on half-edges) satisfies the following second order PDE in the variables $t, \nub$ and $\nuw$:
  \[
    \OpLeft P=    \frac 1 2  (\Lambda^2 P)^2
    + t\left( \nuw+2t^3\left(2\nuw^4 + \nub\nuw^2 + 2\nub^2 + 3\nuw\right)\right)\Lambda^2 P
    +t^5 Q_0,
  \]
  where $\Lambda$ is defined by~\eqref{Lambda-def}, $\OpLeft$ by~\eqref{OpLeft-def}, 
  and 
  \[
    Q_0=2\nuw \left(2  \nuw^{4}+  \nub  \nuw^{2}+2  \nub^{2}  +3  \nuw\right)
    + 2\left(2\nuw^4 + \nub\nuw^2 + 2\nub^2 + 3\nuw\right)^2t^3.
  \]
\end{cor}
\begin{proof} This is obtained by setting $s=0$ in the PDE~\eqref{eq_KP_Ising} satisfied by $I$. A key property is that the operators $\Lambda$ and $\Omega$ do not affect the exponent of $s$.  
\end{proof}

Of course this is a complicated result, compared to the fact that the Ising \gf\ of \emm rooted, planar cubic maps, namely $2D_tP$,
is an explicit algebraic series in $t, \nub, \nuw$ (see~\cite{Ka86,mbm-schaeffer-ising,eynard-book}). In Section~\ref{sec:planar-check}, we derive from the above result a PDE satisfied by the rooted series, and check that the known solution satisfies it.

\begin{remark}
  More generally, extracting from the PDE~\eqref{eq_KP_Ising} the coefficient of $s^g$  gives a PDE for the Ising \gf\ of maps of genus $g$ in terms of the series counting maps of smaller genus.
\end{remark}

% =================================================================
\subsection{The monochromatic white case, and the Goulden-Jackson recursion}
% ==========================================================

We now consider monochromatic white cubic maps, that is, those in which all edges are monochromatic white. Let $M$ be the restriction to such maps of the series $I$ defined by~\eqref{I-def}. We furthermore set $\nuw=1$ in this series, as this variable becomes redundant. Hence $M$ is a series in $t$ and $s$ only.

\begin{cor}\label{cor:white}
  The \gf\ $M$ of cubic (uncolored) maps, labeled on half-edges, counted by edges ($t$) and genus $(s)$ satisfies
  \begin{multline*}
    24 t^{7} \left(t M''+3 M'\right)^{2}
    +4  s \,t^{9}M^{(4)}
    +72  s \,t^{8} M^{(3)}
    +t \left(348 s \,t^{6}+48 t^{6}+12 t^{3}-1\right) M''\\
    +4\left(105 s \,t^{6}+36 t^{6}+9 t^{3}-1\right) M' 
    +3 t^{2} \left(32 s \,t^{3}+8 t^{3}+s +4\right)=0.
  \end{multline*}
  Equivalently, the number $\vec M(n,g)$ of \emm rooted, cubic maps with $3n$ edges and genus $g$ satisfies, for $n\ge 1$ and $g\ge 0$,
  \beq\label{recGJ}
    (n+1)  \vec M_{n,g}=4 n (3n-2)(3n-4) \vec M_{n-2,g-1} +4 \sum_{i+j=n-2 \atop h+k=g} (3i+2)(3 j+2) \vec M_{i,h}\vec M_{j,k},
  \eeq
  with the initial condition $\vec M_{n,g}= \delta_{n,0} \delta_{g,0} -\frac 1 2 \delta_{n,-1} \delta_{g,0}$ for $n\le 0$ or $g<0$.
\end{cor}
The recursion was first established in~\cite{goulden-jackson-KP}. Given the initial condition, the index $i$ ranges from $-1$ to $n-1$ in the summation~\eqref{recGJ}. The above differentialequation can be rewritten in a  more compact way upon introducing  the series $R:=2t^3M'-1/(2t)+t^2$:
\beq\label{R-diff} 
  R  =
 12 t^{6}( R')^2+ 4   s t^{9} R'''+36   s t^{8} R''+t \left(60 s t^{6}-1\right) R'.
\eeq

\begin{proof}
  The contribution of a colored map $\m$ in the Ising series $I$ is of the form $t^{e} \nuw^i \nub^j s^g$, with $e\ge i$. Equality means that the map is white monochromatic, and in this case $j=0$.

  Let us now examine the effect of the operators $\Lambda$ and $\Omega$ on a monomial $m:=t^{\eps + i}\nuw^i \nub^j  s^g$, with $\eps \ge 0$, and in particular on  the exponents of $t$ and $\nuw$. We will see that these operators do not decrease the difference between the exponent of $t$ and the exponent of $\nuw$. Define the following two operators:
  \[
    \Lambda_\circ := 2t^2\nuw^2 (D_t-D_\bullet), \qquad 
    \Omega_\circ:=\frac 1 3 (D_t+D_\circ-D_\bullet-1) \circ \Lambda_\circ -\left( t\nuw (D_t-D_\bullet)\right)^2,
  \]
  and write
  \[
    \Lambda= \Lambda_\circ + \Lambda_1, \qquad   \Omega= \Omega_\circ + \Omega_1.
  \]
  Then
  \[
   \Lambda_\circ(m)=  \Lambda_\circ \left(t^{\eps +i} \nuw^i \nub^j s^g \right)= t^{2+i+\eps}\nuw^{2+i}  \lambda(\nub,s),
  \]
  \[
    \Omega_\circ (m)=  \Omega_\circ \left(t^{\eps +i} \nuw^i \nub^j  s^g\right)= t^{2+i+\eps}\nuw^{2+i}\omega( \nub,s),
  \]
  for some functions $\lambda$ and $\omega$ that do not involve $t$ nor $\nuw$.  On the other hand, in all monomials occurring in $\Lambda_1(m)$ and $\Omega_1(m)$, the exponent of $t$ exceeds the exponent of $\nuw$ by at least $1+\eps$.

  Hence, if we extract from~\eqref{eq_KP_Ising} the monomials where $t$ and $\nuw$ have the same exponent, we obtain
  \[
    \Omega_\circ M_\circ= \frac s{12} \Lambda_\circ^4 M_\circ+ \frac 1 2 \left( \Lambda_\circ^2 M_\circ\right)^2 + t\nuw(1+4 t^3 \nuw^3) \Lambda_\circ^2 M_\circ+ t^5 \nuw^5 (4+s) + 8t^8 \nuw^8 (1+4s),
  \]
  where $M_\circ$ is obtained from $I$ by extracting monomials where $t$ and $\nuw$ have the same exponent.
  Equivalently, $M_\circ$ is the series $M$ evaluated at $t\nuw$. This series does not involve the variable~$\nub$: hence, in the above identity, we can replace $\Lambda_\circ$ and $\Omega_\circ$, respectively, by
  \[
    2t^2\nuw^2 D_t \qquad \text{and} \qquad
    \frac 1 3 \left(D_t+D_\circ-1\right) \circ (2t^2\nuw^2 D_t) -\left( t\nuw D_t\right)^2. 
  \]
  Observe further that the operators $D_t$ and $D_\circ$ act in the same way on monomials in which $t$ and $\nuw$ have the same exponent, and thus on monomials of $M_\circ$. Hence we can replace $D_\circ$ by $D_t$ above. Setting finally $\nuw=1$ gives
  \[
    \overline\Omega_\circ M= \frac s{12} \overline\Lambda_\circ^4 M+ \frac 1 2 ( \overline\Lambda_\circ^2 M)^2 + t(1+4 t^3) \overline\Lambda_\circ^2 M+ t^5  (4+s) + 8t^8 (1+4s),
  \]
  with
\[ 
    \overline\Lambda_\circ :=
    2t^2 D_t \qquad \text{and} \qquad
    \overline\Omega_\circ:=
    \frac 1 3 (2D_t-1) \circ (2t^2D_t)-( tD_t)^2.
  \]
  This is the fourth order differential equation announced in the corollary. The differential equation~\eqref{R-diff} for the series $R$ follows, and  the recursion~\eqref{recGJ} is obtained by writing
  \[
    M= \sum_{n\ge 1, g\ge 0} \frac{\vec M_{n,g}}{6n} t^{3n}s^g,
    \qquad
    \text{hence} \qquad R=\sum_{n\ge -1, g\ge 0} \vec M_{n,g}t^{3n+2}s^g,
  \]
  and extracting the coefficient of $t^{3n+2}s^g$ in~\eqref{R-diff}.
\end{proof}

%\begin{remark}  Goulden and Jackson also derived the above recursion from the first KP equation~\eqref{eq_KP}. But they obtained it by forbidding, in bipartite maps, {white}  vertices of degree $2$ ($p_2=0$), while we obtain it, essentially, by forbidding black vertices of degree $3$.\label{rem:gj}\end{remark}

\begin{remark}
  One can combine specializations. For instance, a DE for the generating function $M_0$ of planar white cubic maps is obtained by setting $s=0$ in the DE of Corollary~\ref{cor:white}:
  \[
    24 t^{7} \left(t M_0''+3 M_0'\right)^{2}
       +t \left(48 t^{6}+12 t^{3}-1\right) M_0''
    +4\left(36 t^{6}+9 t^{3}-1\right) M_0' 
    +12 t^{2} \left(2 t^{3}+1\right)=0.
  \]
  Its solution, given by
  \[
    2t M_0'=\sum_{n\ge 1} \frac{2\cdot 8^n}{(n+1)(n+2)} \binom{3n/2}{n} t^{3n},
  \]
  has been known since the early work of Mullin, Nemeth  and Schellenberg~\cite{mullin-nemeth-schellenberg}. The above series satisfies a polynomial equation of degree $3$.
  \end{remark}

% ==================================================
\subsection{The unicellular case}
% ==================================================
The planar case studied in Section~\ref{sec:planar} corresponds to maps having a maximal number of faces, given their edge number. Here we study the other extreme, with maps having a single face (also called \emm unicellular,), or equivalently maximal genus. Let $U$ be the restriction to unicellular maps of the series $I$ defined in~\eqref{I-def}. We furthermore set $s$ to $1$, as the edge number and the genus are then directly related by $\ee(\m)= 3(2\gen(\m)-1)$. In particular, there are no unicellular cubic maps in genus $0$.

We obtain a first \emm linear, PDE for $U$ by specializing our general result to the unicellular case, and then a simpler one by adapting  a combinatorial construction of monochromatic unicellular maps due to Chapuy~\cite{chapuy-new}. The latter PDE yields
explicit hypergeometric expressions for the number of Ising maps of size $3n$ having only a fixed number $d$ of bicolored edges, for small values of $d$. We conjecture that such formulae exist for any $d$. We also conjecture hypergeometric expressions for maps with $d$ monochromatic edges, for $d$ small.

\subsubsection{Two PDEs}
We begin with a PDE derived from~\eqref{eq_KP_Ising}.

\begin{cor}\label{cor:uni} 
  The Ising \gf\ $U$ of unicellular cubic maps satisfies the following fourth order \emm linear, PDE in the variables $t, \nub$ and $\nuw$:
  \[
    \OpLeft U=    \frac 1 {12}  \Lambda^4 U +t^5\left( \nuw^5+2\nuw^2+\nub\right),
  \]
  where $\Lambda$ is defined by~\eqref{Lambda-def} and $\OpLeft$ by~\eqref{OpLeft-def}.
\end{cor}

\begin{proof}
  The contribution of a bicolored map $\m$ in the Ising series $I$ is of the form $t^{3n} s^g \nuw^i \nub^j$, and by Euler's relation, the number of faces of $\m$ is then $2+n-2g$. Since there is at least one face,  one always has $n\ge 2g-1$, and  equality holds for unicellular maps only.

  Let us now examine the effect of the operators $\Lambda$ and $\Omega$ on a monomial $t^{\eps + 3(2g-1)} s^g \nuw^i \nub^j$, with $\eps \ge 0$, and in particular on the exponents of $t$ and $s$. We find:
  \[
    \Lambda \left(t^{\eps +3( 2g-1)} s^g \nuw^i \nub^j\right)= t^{2+\eps +3( 2g-1)} s^g \lambda(\nuw, \nub),
  \]
  \[
    \Omega \left(t^{\eps +3( 2g-1)} s^g \nuw^i \nub^j\right)= t^{2+\eps +3( 2g-1)} s^g \omega(\nuw, \nub),
  \]
  for some functions $\lambda$ and $\omega$ that do not involve $t$ nor $s$. In particular,
  \[
    s\Lambda^4 \left(t^{\eps +3( 2g-1)} s^g \nuw^i \nub^j\right)= t^{8+\eps +3( 2g-1)} s^{g+1} \lambda_4(\nuw, \nub)=t^{2+ \eps+3( 2(g+1)-1)} s^{g+1} \lambda_4(\nuw, \nub).
  \]
  Hence, if we extract from~\eqref{eq_KP_Ising} the monomials of the form $t^{2+3( 2g-1)} s^g\nuw^i \nub^j$, we find that the terms involving $\Lambda^2I$ do not contribute, that $\Omega I$ contributes $\Omega U$, while $s\Lambda^4I$ contributes~$s\Lambda^4 U$. This gives the announced equation on $U$.
\end{proof}

A simpler PDE, also linear but of order $3$ only, can be written for the series $U$ by adapting to unicellular Ising cubic maps a construction designed by Chapuy for monochromatic unicellular cubic maps~\cite[Sec.~6.2]{chapuy-new}.

\begin{prop}\label{prop:PDE-U-simple}
   The Ising \gf\ $U$ of unicellular cubic maps is characterized by the following third order  linear PDE in the variables $t, \nub$ and $\nuw$:
  \[ 
  \left(6 + 2 D_t - \Fw^3 -\Fb^3\right)U=0,
\]
together with the initial condition for maps with $3$ edges (genus $1$):
\[
U_1:= t^3 [t^3] U=\frac{t^3}{6}(\nuw^3+\nub^3+{2}),
\]
where the operators $\Fw$ and $\Fb$ are defined by
\[
  \Fw= 2t^2\left((\nub+\nuw^2)
    \left(D_t -D_\circ -D_\bullet
    \right)+(\nub^2+\nuw)
     \partial_\bullet    +(1+\nuw^3)
    \partial_\circ\right),
\]
and symmetrically
\[
  \Fb= 2t^2\left((\nuw+\nub^2)\left(D_t -D_\circ -D_\bullet
    \right)+(\nuw^2+\nub)
    \partial_\circ  +(1+\nub^3)
   \partial_\bullet\right).
\]
Equivalently, if $U_g$ is the contribution in $U$ of maps of genus $g$, that is, $U_g=t^{3(2g-1)}[t^{3(2g-1)}] U$, we have, for $g\geq 2$:
\beq\label{Ug-rec}
  2g U_g=\frac{1}{6}(\Fb^3+\Fw^3)U_{g-1}.
\eeq
\end{prop}
\begin{proof}
  Let us call \emm precubic, a map with vertices of degree $1$ and $3$ only. Chapuy's construction~\cite{chapuy-new} implies that there exists a bijection between
  \begin{itemize}
  \item 
    unicellular cubic (Ising) maps of genus~$g$ with a distinguished \emm trisection,,
    \item and unicellular precubic (Ising) maps of genus $g-1$ having exactly~$3$ leaves, of the same color.
  \end{itemize}
  This bijection preserves the number of edges of each type.  Moreover, a unicellular Ising map of genus $g$ has exactly $2g$ trisections~\cite[Lem.~3]{chapuy-new}. Hence the term $2g U_g$ on the left-hand side of~\eqref{Ug-rec} counts unicellular cubic Ising maps of genus $g$ with a distinguished trisection. We will show that $\frac 1 6 \Fw^3 U_{g-1}$  counts unicellular precubic Ising maps of genus $g-1$ having exactly~$3$ white leaves, and this will prove~\eqref{Ug-rec} by a symmetry argument.

  Consider  a unicellular precubic Ising map $\m$, with Ising weight
      \[
 W:= \frac{t^{e}}{(2e)!}
  \nuw^{e^\circ}\nub^{e^\bullet},
\]
and construct a unicellular precubic map having a marked white leaf as follows: choose an edge $e$ in $\m$,  create a vertex~$v$ in the middle of $e$ and insert a new edge starting from $v$ and ending at a white leaf, lying on one of the two sides of $e$.  The color of $v$ can be chosen in two possible ways. Then
the sum of the Ising weights of all maps obtained by the above construction is
\[
2  \frac{t^{e+2}}{(2e+4)!}  \nuw^{e^\circ}\nub^{e^\bullet}
\left( (e-e^\circ-e^\bullet) (\nub+\nuw^2)
  + e^\bullet (\nub+\nuw/\nub)
  + e^\circ ( 1/\nuw+\nuw^2)
    \right) \binom{2e+4}{4}4! = \Fw( W).
  \]
  In this expression, one can read off the type of the edge $e$ (first bichromatic, then black, then white), the color of $v$ (first black, then white), and the choice of $4$ half-edge labels, one at the new leaf, the other $3$ around $v$.

 To obtain a precubic unicellular Ising map with exactly $3$ leaves, all of them white, we start from a \emm cubic, unicellular Ising map, apply three times the above construction, and forget the order in which the leaves were created (this results in a factor $1/6$). This concludes the proof of~\eqref{Ug-rec}.
 Now recall that unicellular cubic maps of genus $g$ have $3(2g-1)$ edges. Summing~\eqref{Ug-rec} over $g$ then gives the first equation of the proposition. The value of $U_1$ can be obtained from the central column of Figure~\ref{fig:small-ising}. 
\end{proof}

%=============================================
\subsubsection{Explicit coefficients}
% =============================================

Let us denote by $U_{n,k,\ell}$ the coefficient of $t^{3n} \nub^k \nuw^\ell$ in the series~$U$.  Recall that $U_{n,k,\ell}$ is zero if $k+\ell >3n$, but also if $k$ and $\ell$ do not  differ by a multiple of~$3$. Moreover, the only non-zero coefficients $U_{n,k,\ell}$ with $k+\ell =3n$ are $U_{n,3n,0}$ and the symmetric term $U_{n,0,3n}$, which count monochromatic unicellular cubic maps. These numbers are known to have a simple expression: for $n=2g-1$, 
\[
U_{n,3n,0} = \frac{(6g-4)!}{12^{g} g! (3g-2)!}.
\]
A rooted version of this result, that is, the above number multiplied by $6(2g-1)$
can be found explicitly for instance in~\cite[Cor.~8]{chapuy-new}.  But it is also equivalent to a special case of an older result due to   Walsh and Lehman~\cite[Eq.~(9)]{walsh-lehman-I}.

It seems that this hypergeometric pattern persists for maps with ``many'' monochromatic edges.

\begin{conjecture}\label{conj:numbers-uni}
  For fixed $i, j\ge 0$, there exists a rational function $R_{i,j}(g)$ in $g$
  such that for $n=2g-1$, 
  \beq\label{many-HG}
        U_{n,3n-i,j} = R_{i,j}(g) \frac {(6g)!}{12^g g!(3g)!}.
  \eeq
  These numbers vanish when $i+j$ is not a multiple of $3$ or when $j>i$.
  \end{conjecture}

 For instance,   for $i+j=3$,  we have
\[
  U_{n,3n-2,1}=0 \quad \text{and} \quad U_{n, 3n-3,0}=\frac 2 3\cdot \frac{(6g-3)!}{{12}^{g} g! (3g-2)!}
\]
while for $i+j=6$,
\[
  U_{n,3n-4,2}                                                                              =\frac{6\cdot (6g-6)!}{12^{g} (g-1)! (3g-3)!}, \quad 
  U_{n,3n-5,1}   = (12 g^2-18g+5) \cdot \frac{6\cdot (6g-6)!}{12^{g} g! (3g-3)!},
\]
\[
  U_{n,3n-6,0}  =  (4g-5)(12 g^2-19g+6)\cdot \frac{2\cdot (6g-6)!}{12^{g} g! (3g-3)!}.
\]

\begin{proof}[Some evidence for Conjecture~\ref{conj:numbers-uni}] We have proved this conjecture for $i+j\le 24$ and it would not be hard to push this further. Let us explain how this proof works.
  The recursion~\eqref{Ug-rec} translates into a recurrence relation that gives $U_{n,3n-i,j}$ in terms of $22$ coefficients  $U_{n-2,3n-6-i',j'}$ where $i'+j'\le i+j$ and $i'\le i$ if $i'+j'= i+j$; see our \Maple\ session for details. 
The term involving $U_{n-2,3n-6-i,j}$ can be written explicitly, for $n=2g-1$:
\[
  U_{n,3n-i,j}= \frac{2(6g-j-5)(6g-j-7)(6g-j-9)}{3g} U_{n-2,3n-6-i,j}+ \cdots
\]
This allows us to prove expressions of the form~\eqref{many-HG} by increasing induction on $i+j$, and, for $i+j$ fixed, increasing induction on~$i$, as follows:
\begin{itemize}
\item  we generate the numbers $U_{n,3n-i,j}$ for many values of $n$,
\item then we guess from these values   a hypergeometric expression of the form~\eqref{many-HG},
  \item we  finally prove this expression by checking that it satisfies  the above recurrence relation.
\end{itemize}
The form of this relation does not seem to imply directly our conjecture, however. 

\smallskip
Another evidence for the above conjecture is the case $i-j=1$. That is, only one of the~$3n$ edges is bicolored. This edge joins two unicellular maps, one black with $3n-i$ edges, one white with $j=i-1$ edges.  Splitting this bicolored edge in its middle gives two unicellular precubic monochromatic maps, rooted canonically at their unique vertex of degree $1$, having respectively $3n-i+1$ and $i$ edges. Euler's relation then forces $i$ to be of the form $6\ell-1$. Using Chapuy's results on (rooted)  precubic maps~\cite{chapuy-new}, one obtains
   \[
    U_{n,3n-6\ell +1, 6\ell-2}
    =\frac{(6\ell-2)!}{\ell !(3\ell -1)!}  \cdot \frac{(6g-6\ell-2)!}{12^g (g-\ell)!  (3g-3\ell-1)! }.
  \]
  This proves the conjecture in the case $i-j=1$.
  \end{proof}

    Above we have considered Ising maps with many monochromatic edges. At the other end of the scale we have bipartite maps, counted by coefficient of $t^{n} \nub^0 \nuw^0$. It seems that we have no direct access to these numbers, even though they have a simple expression:
\[
  U_{n,0,0}= 2 \cdot \frac{(3g-2)!(2g-3)!}{3^{g}  g! (g-1)! (g-2)!}.
\]
This is equivalent to a special case of~\cite[Thm.~2.1]{goupil-schaeffer}. We  predict other similar hypergeometric formulae for unicellular Ising map with few monochromatic edges, \emm e.g.,, for $n=2g-1$: 
\begin{align*}
  U_{n,0,3}
  &= \frac 1 4  (3g-5) \cdot \frac{(3g-4)!(2g-3)!}{3^{g-1}  (g-1)! (g-2)!^2} , \\
  U_{n,1,1}
  &=   \frac 1 2 \cdot \frac{(3g-5)!(2g-3)! }{3^{g-3}  (g-1)! (g-2)!(g-3)!} , \\
  U_{n,1,4}
 & =   \frac 1 {32} (18g^3-75g^2+75g-2) \cdot \frac{(3g-4)!(2g-3)!}{3^{g-2}  g! (g-2)!(g-3)!} , \\
  U_{n,2,2}
  &=    \frac 1 {16}(54g^3-225g^2+231g-4) \cdot \frac{(3g-5)!(2g-3)! }{3^{g-2} g! (g-2)!(g-3)!}.
\end{align*}

\begin{conjecture}
    For fixed $i, j\ge 0$,  there exists a rational function $Q_{i,j}(g)$ in $g$
  such that for $n=2g-1$, 
  \[
         U_{n,i,j} = Q_{i,j}(g) \frac {(2g)!(3g)!}{3^g g!^3}.
  \]
  These numbers vanish when $i-j$ is not a multiple of $3$. 
\end{conjecture}

%%%%%%%%%%%%%%%%%%%%%%%%%%%%%%%%%%%%%%%%%%%%%% 
\section{Inequalities}
\label{sec:inequalities}
%%%%%%%%%%%%%%%%%%%%%%%%%%%%%%%%%%%%%%%%%%%%%% 

In this section, we prove Corollary~\ref{thm_bounds}, which gives a lower bound on the numbers $\vec I_{n,g}(\nub, \nuw)$ counting rooted unlabeled Ising maps, for $\nub$ and $\nuw$ positive.

Let $\semiring$ be the sub-semiring of $\mathbb{R}_{\geq 0}[s,\nub,\nuw][[t]]$ generated by all monomials $t^{m}s^g\nub^a\nuw^b$  such that $m\geq a+b$. The proof of Corollary~\ref{thm_bounds} relies on the PDE~\eqref{eq_KP_Ising} and exploits the fact that $I\in\semiring$, so that $(D_t-D_\circ-D_\bullet)I$ has non-negative coefficients. However, since our final inequality implies a $\min$ function, which is not compatible with differential operators, the formalization of the proof is a bit heavy.

We consider here  linear differential operators in $t$, $\nub$ and $\nuw$, whose coefficients  are polynomials in $s, \nub, \nuw$ and $t$.
This includes the identity operator, so that a monomial in $s$, $t$, $\nub$ and $\nuw$ is seen as an operator that acts by multiplication on power series.
We define an order relation on such operators as follows: we say that $\Phi\succeq \Psi$ if 
\[
  (\Phi-\Psi)(\semiring)\subset\semiring.
\]
Clearly, if $a\in \rs_{\ge 0}$ and $\Phi\succeq \Psi$, then $a\Phi\succeq a\Psi$ and $-a\Phi\preceq -a\Psi$. We also have the following properties.

\begin{lemma}\label{lem_order_composition}
  The order relation $\succeq$ is compatible with addition and with composition, in the following sense: if $\Xi\succeq \Upsilon$, $\Xi\succeq 0$
 and $\Phi\succeq \Psi\succeq 0$, then 
\beq\label{eq_order_composition}
\Xi\circ \Phi \succeq \Upsilon\circ \Psi.
\eeq
Also, given nonnegative integers $m,a,b$ satisfying $m\geq a+b$ and two operators $\Phi,\Psi\succeq 0$ with $\Phi$ homogeneous of first order, one has
\beq\label{eq_order_scalar}
\Phi \circ (t^{m}\nub^a\nuw^b\, \Psi)
\succeq t^{m}\nub^a\nuw^b\, \Phi\circ \Psi.
\eeq
\end{lemma}

\begin{proof}
Compatibility with addition is clear: if $\Phi_1 \succeq \Psi_1$ and  $\Phi_2 \succeq \Psi_2$, then  $\Phi_1+\Phi_2 \ge \Psi_1+\Psi_2$. Now take $\Xi\succeq \Upsilon$ with $\Xi\succeq 0$, and $\Phi\succeq \Psi\succeq 0$.
It is direct by composition that
\[
  \Xi\circ(\Phi-\Psi)(\semiring)\subset \semiring,
\] hence $\Xi\circ \Phi\succeq \Xi\circ \Psi$. Similarly, one can prove that $\Xi\circ \Psi\succeq \Upsilon\circ \Psi$, and Property~\eqref{eq_order_composition} follows by transitivity.

Now we prove Property~\eqref{eq_order_scalar}. Since $\Phi$ is a derivation, we have
\[
  \Phi \circ (t^{m}\nub^a\nuw^b\, \Psi)= t^{m}\nub^a\nuw^b\, \Phi\circ \Psi+\Phi(t^{m}\nub^a\nuw^b)  \Psi.
\]
The quantity $\Phi(t^{m}\nub^a\nuw^b)$ belongs to $\semiring$ but
is also an operator acting by multiplication, and as such we have $\Phi(t^{m}\nub^a\nuw^b)\succeq 0$. By Property~\eqref{eq_order_composition} we then have $\Phi(t^{m}\nub^a\nuw^b)\Psi\succeq 0$, and the proof follows.
\end{proof}

\begin{lemma}\label{lem_bounds_operators}
Let $\Lambda$ and $\OpLeft$ be the operators defined in~\eqref{Lambda-def} and~\eqref{OpLeft-def}, respectively. Then we have the following inequalities:
\begin{align}
&\Lambda^k\succeq 2^kt^{2k}\left(\sum_{j=0}^k \binom{k}{j} \nuw^{2j}\nub^{k-j}(D_t-D_\bullet)^j(D_t-D_\circ)^{k-j}\right)\succeq 0,\label{xa}\\
&t^2\nub\nuw\OpLeft\preceq\frac{4t^4}{3} \left(\nuw^3\nub+\nuw\nub^2+\nuw^2+\nub\right)(D_t+2)\circ D_t.\label{xb}
\end{align}
\end{lemma}

\begin{proof}

Let us start with some observations:
\begin{itemize}
\item $D_t,D_\circ,D_\bullet\succeq 0$,
\item $D_t-D_\circ-D_\bullet\succeq 0$,
\item $t^2\nuw^2,t^2\nub,t^2\nub\nuw\succeq 0$,
\item $t^2\nuw\partial_\bullet, t^2\partial_\circ\succeq 0$.
\end{itemize}
  We first prove~\eqref{xa} for $k=1$. By definition~\eqref{Lambda-def} of $\OpRight$, and thanks to the observations above, we have
\begin{align}
 \OpRight&=2t^2\left((\nuw^2+\nub)D_{t}+ \nuw(1-\nub\nuw)\partial_{\bullet}+(1-\nub\nuw)\partial_{\circ}\right) \label{Lambda-def2}\\
 &=2t^2\left(\nuw^2(D_t-D_\bullet)+\nub(D_t-D_\circ)+ \nuw\partial_{\bullet}+\partial_{\circ}\right) \nonumber\\
 &\succeq2t^2\left(\nuw^2(D_t-D_\bullet)+\nub(D_t-D_\circ)\right) \succeq 0,\nonumber
\end{align}
which is precisely~\eqref{xa}. Let us now proceed by induction. Fix $k\geq 1$ and assume that~\eqref{xa} holds for $k$. Since $\OpRight\succeq 0$, thanks to Property~\eqref{eq_order_composition}, we immediately have $\OpRight^k\succeq 0$ and we can also write:
\begin{align*}
\OpRight^{k+1}&=\OpRight\circ\OpRight^{k}\\
&\succeq 2t^2\left(\nuw^2(D_t-D_\bullet)+\nub(D_t-D_\circ)\right)\circ\left( 2^kt^{2k}\left(\sum_{j=0}^k \binom{k}{j} \nuw^{2j}\nub^{k-j}(D_t-D_\bullet)^j(D_t-D_\circ)^{k-j}\right)\right)\\
&= \sum_{j=0}^k \binom{k}{j} 2t^2\left(\nuw^2(D_t-D_\bullet)+\nub(D_t-D_\circ)\right)\circ\left( 2^kt^{2k} \nuw^{2j}\nub^{k-j}(D_t-D_\bullet)^j(D_t-D_\circ)^{k-j}\right)\\
&\succeq 2^{k+1}t^{2(k+1)}\left(\sum_{j=0}^{k+1} \binom{k+1}{j} \nuw^{2j}\nub^{k+1-j}(D_t-D_\bullet)^j(D_t-D_\circ)^{k+1-j}\right),
\end{align*}
where in the last inequality we used Property~\eqref{eq_order_scalar} and the commutation of $(D_t-D_\bullet)$ and $(D_t-D_\circ)$.

Now, we turn to~\eqref{xb}. We examine the various terms in the expression~\eqref{OpLeft-def} of $\OpLeft$. It follows from~\eqref{Lambda-def2} that
\[
 0 \preceq  \OpRight\preceq 2t^2\left((\nuw^2+\nub)D_t+ \nuw\partial_{\bullet}+\partial_{\circ}\right).
\]
It is also direct that 
\[
  D_t+D_{\circ}-D_{\bullet}-1\preceq 2D_t
.\]
Hence, by Property~\eqref{eq_order_composition}, we have:
\begin{align}\label{ineq_om1}
t^2\nub\nuw\left(D_t+D_{\circ}-D_{\bullet}-1\right)\circ \OpRight&\preceq 2t^2\nub\nuw D_t\left(2t^2\left((\nuw^2+\nub)D_t+ \nuw\partial_{\bullet}+\partial_{\circ}\right)\right) \notag \\
&={2} t^2D_t \left(2t^2\left((\nuw^3\nub+\nub^2\nuw)D_t+ \nuw^2D_\bullet+\nub D_\circ\right)\right) \notag \\
&\preceq {2} t^2D_t \left(2t^2(\nuw^3\nub+\nub^2\nuw+\nuw^2+\nub)D_t\right) \notag \\
                                                                 &= 4t^4(\nuw^3\nub+\nub^2\nuw+\nuw^2+\nub) (D_t+2)\circ D_t.
\end{align}
Finally, one can check that $t \nuw D_t + t (1-\nub \nuw)\partial _{\bullet}\succeq 0$, hence by Property~\eqref{eq_order_composition}:
\beq\label{ineq_om2}
t^2\nub\nuw\big( t \nuw D_t + t (1-\nub \nuw)\partial _{\bullet}\big)^2\succeq 0.
\eeq
Combining inequalities~\eqref{ineq_om1} and~\eqref{ineq_om2} in~\eqref{OpLeft-def} yields~\eqref{xb}.
\end{proof}

The following lemma is a rather direct consequence of Lemma~\ref{lem_bounds_operators}.

\begin{lemma}\label{lem_final_ineq}
Given $x,y>0$, let $\Xi$ be the operator defined on series of $\semiring$ that sets $\nub$ to $x$ and $\nuw$ to~$y$. Then, for $F\in\semiring$, the following inequalities hold coefficientwise in $s$ and $t$:
\beq\label{eq_xi}
\Xi(\OpRight^k F)\geq 2^{k}t^{2k}\min(y^2,x)^k D_t^{k}(\Xi F) ,
\eeq
\beq\label{eq_xi_2}
\Xi(\OpLeft F) \leq \frac{ 4 t^2}{3} \left(y^2+x+\frac{y}{x}+\frac{1}{y}\right)(D_t+2)(D_t(\Xi F)).
\eeq 
\end{lemma}

\begin{proof}
We first observe that if $\Phi \succeq \Psi$ and $F\in \semiring$, then $\Xi\Phi(F) \ge \Xi\Psi(F)$ coefficientwise in $s$ and~$t$. Hence by~\eqref{xa}, the following inequalities hold coefficientwise in $s$ and $t$:
\begin{align*}
\Xi(\OpRight^k F)&\geq \Xi\left(2^{k}t^{2k}\left(\sum_{j=0}^{k} \binom{k}{j} \nuw^{2j}\nub^{k-j}(D_t-D_\bullet)^j(D_t-D_\circ)^{k-j}\right)F\right)\\
&=2^{k}t^{2k}\sum_{j=0}^{k} \binom{k}{j} (y^{2j}x^{k-j})\, \Xi\left((D_t-D_\bullet)^j(D_t-D_\circ)^{k-j}F\right)\\
&\geq 2^{k}t^{2k}\min(y^2,x)^k\sum_{j=0}^{k} \binom{k}{j}\, \Xi\left((D_t-D_\bullet)^j(D_t-D_\circ)^{k-j}F\right)\\
&=2^{k}t^{2k}\min(y^2,x)^k\, \, \Xi\left((2D_t-D_\circ-D_\bullet)^k F\right),
\end{align*}
because the operators $(D_t-D_\bullet)$ and $(D_t-D_\circ)$ commute. We then conclude the proof of~\eqref{eq_xi} using $D_t\succeq D_\circ+D_\bullet$ and Lemma~\ref{lem_order_composition}.

The inequality~\eqref{eq_xi_2} follows  from~\eqref{xb} in a straightforward way.
\end{proof}

We are now ready to prove Corollary~\ref{thm_bounds}.

\begin{proof}[Proof of Corollary~\ref{thm_bounds}]
  Let $\nub, \nuw>0$. First, note that the \gf\ of rooted unlabeled Ising cubic maps is $\vec I:=\sum_{n,g} \vec I_{n,g}(\nub,\nuw) t^{3n} s^g=2D_t I$.
    Also, all monomials in  $I$ are of the form $t^{3n}\nub^a\nuw^bs^g$ with $3n\geq a+b$ (because a map contributing this monomial has $3n$ edges, $a+b$ of which are  monochromatic). Hence, $I\in\semiring$ and we can apply Lemma~\ref{lem_final_ineq} to obtain
\begin{align}
  &\Lambda^kI\geq 2^{k}t^{2k}\min(\nuw^2,\nub)^k D_t^{k}I
    = 2^{k-1}t^{2k}\min(\nuw^2,\nub)^k D_t^{k-1}\vec I,\label{ineq_r1}\\
  &\OpLeft I \leq
    \frac{  4t^2}{3} \left(\nuw^2+\nub+\frac{\nuw}{\nub}+\frac{1}{\nuw}\right)(D_t+2) (D_t I) =
    \frac{
      % {\color{red} 4}
    2t^2}{3} \left(\nuw^2+\nub+\frac{\nuw}{\nub}+\frac{1}{\nuw}\right)(D_t+2)\vec I,\label{ineq_r2}
\end{align}
where it is implied (here and in the rest of the proof) that the inequalities hold for $\nub,\nuw>0$, coefficientwise in $s$ and $t$.

Moreover, since the last two terms of the RHS in the PDE~\eqref{eq_KP_Ising} satisfied by $I$ have nonnegative coefficients, we have the following inequality:
\[
  \OpLeft I\geq   \frac s {12}\Lambda ^4I
  + \frac 1 2  (\Lambda^2 I)^2.
\]
Combined with \eqref{ineq_r1} and~\eqref{ineq_r2}, it yields
\[
  % {\color{red} 2}
  \left(\nuw^2+\nub+\frac{\nuw}{\nub}+\frac{1}{\nuw}\right)(D_t+2)\vec I\geq t^6 \min(\nuw^2,\nub)^4\left( s D_t^3\vec I+3 \left(D_t \vec I\right)^2 \right).
\]
Now, extracting the coefficient of $t^{3n}s^g$ above yields
\[
  \left(\nuw^2+\nub+\frac{\nuw}{\nub}+\frac{1}{\nuw}\right)(3n+2)\vec I_{n,g}\geq
    %     \frac{27}{\color{red} 2}
  27\min(\nuw^2,\nub)^4\left((n-2)^3\vec I_{n-2,g-1}+\sum_{i+j=n-2\atop h+k=g}i\vec I_{i,h}(\nub,\nuw)j\vec I_{j,k}(\nub,\nuw)\right),
\]
hence for $n\ge 2$,
\[
  n\vec I_{n,g}\geq \frac{27(n-2)^3n}{
    % {\color{red} 2}
    (3n+2)n^3}\frac{\min(\nuw^2,\nub)^4}{\left(\nuw^2+\nub+\frac{\nuw}{\nub}+\frac{1}{\nuw}\right)}\left(n^3\vec I_{n-2,g-1}+\sum_{i+j=n-2\atop h+k=g}i\vec I_{i,h}(\nub,\nuw)j\vec I_{j,k}(\nub,\nuw)\right),
\]
and since $\frac{27(n-2)^3}{(3n+2)n^2}\geq 1$ for $n\geq  5$,  Corollary~\ref{thm_bounds} follows.
\end{proof}

%%%%%%%%%%%%%%%%%%%%%%%%%%%%%%%%%%%%%%%%%%%%%%%%%%%%%%%%%%%%%%%%%%%%%%%
\section{Rooted unlabeled Ising maps, and a verification in genus $0$}
\label{sec:rooted}
%%%%%%%%%%%%%%%%%%%%%%%%%%%%%%%%%%%%%%%%%%%%%%%%%%%%%%%%%%%%%%%%%%%%%%%

Since many results on Ising maps deal with rooted, rather than labeled, maps,  we now derive from the PDE~\eqref{eq_KP_Ising} satisfied by the series $I$ another PDE satisfied by the \gf \ of \emm rooted, Ising cubic maps.
Then, we check this new equation  (or more precisely a smaller equation derived similarly from the $s=0$ version of~\eqref{eq_KP_Ising}) in the planar case, starting from the known rational parametrization of the rooted planar Ising series.

%===================================================
\subsection{A {partial} differential equation for rooted maps}
\label{sec:rootedPDE}
% ===================================================

We first get rid of the periodicity in $t$ by introducing $\widetilde{I}(t,\nub,\nuw,s):=I(t^{1/3},\nub,\nuw,s)$, where we have made all variables explicit. Then, the \gf \ of rooted unlabeled Ising cubic maps counted by the number of edges divided by $3$,  the numbers of black and white monochromatic edges and the genus) is $\vec{I}:=6D_t\widetilde{I}$. The notation $\vec I$ was used in the previous section to denote the same series,
% but
with $t$ replaced by $t^3$. We hope that this will not cause any confusion.

\begin{prop}
  The \gf\ $\vec{I}(t,\nub,\nuw,s)$ of rooted unlabeled Ising cubic maps satisfies an explicit partial differential equation in the variables $t, \nub, \nuw$, of order $7$ and degree $3$, with coefficients in $\qs[t,\nub,\nuw,s]$.
\label{prop: PDE rooted}
\end{prop}

We do not write down explicitly this equation as it is quite large, but it is available in the {\sc Maple} companion file. The following proof explains how to construct it.

\begin{proof}
  We start from the PDE
  \eqref{eq_KP_Ising} satisfied by $I$:
\[
    \OpLeft I=   \frac s {12}\Lambda ^4I
    + \frac 1 2  (\Lambda^2 I)^2
    + t\left( \nuw+2t^3\left(2\nuw^4 + \nub\nuw^2 + 2\nub^2 + 3\nuw\right)\right)\Lambda^2 I
    +t^5 Q.
\]
When we replace the operators $\OpLeft$ and $\Lambda$ by~\eqref{OpLeft-def} and~\eqref{Lambda-def}, respectively, we observe that  a global factor $t^2$ comes out. We divide the PDE by this factor, replace $I(t, \nub,\nuw,s)$ by $\widetilde I(t^3, \nub, \nuw,s)$ and apply repeatedly the following identity:  if $J(t)=\widetilde J(t^3)$, for some function $\widetilde J(\tau)$, then
 \[
 D_t\left(t^k\,  J(t)\right)=\tau^{k/3}(k \widetilde J(\tau)+3\tau\partial_{\tau}\widetilde J(\tau)),
\]
to obtain an equation satisfied by $\widetilde{I}= \widetilde I(t, \nub,\nuw,s)$:
\beq
\label{eq: KP Ising aperiodic}
    \widetilde{\OpLeft} \widetilde{I}=   \frac s {12}t^2\widetilde{\Lambda} _4\widetilde{I}
    + \frac 1 2  t^2\left(\widetilde{\Lambda}_2 \widetilde{I}\right)^2
    + t\left( \nuw+2t\left(2\nuw^4 + \nub\nuw^2 + 2\nub^2 + 3\nuw\right)\right)\widetilde{\Lambda}_2 \widetilde{I}
    + t \widetilde{Q},
\eeq
with 
\allowdisplaybreaks
\begin{align*}
  \widetilde{\OpRight}&=2\left(3(\nuw^2+\nub)D_t+ \nuw(1-\nub\nuw)\partial_\bullet+(1-\nub\nuw)\partial_\circ\right),
  \\
  \widetilde{\OpLeft}&=\frac{1}{3}\left(3D_t+D_\circ-D_\bullet+1\right)\circ\widetilde{\OpRight}
                       - \left( \nuw (3D_t+1) +  (1-\nub \nuw)\partial_\bullet\right)\circ\left( 3 \nuw D_t +  (1-\nub \nuw)\partial_\bullet\right),
  \\
\widetilde{\OpRight}_2&=\left( \widetilde{\OpRight}+4(\nuw^2+\nub)\right)\circ\widetilde{\OpRight} ,
  \\
\widetilde{\OpRight}_4&=\left(\widetilde{\OpRight}+12(\nuw^2+\nub)\right)\circ\left(\widetilde{\OpRight}+8(\nuw^2+\nub)\right)\circ\left(\widetilde{\OpRight}+4(\nuw^2+\nub)\right)\circ\widetilde{\OpRight},
\end{align*}
and 
\begin{multline*}
 \widetilde{Q}=
  2\nuw \left(2  \nuw^{4}+  \nub  \nuw^{2}+2  \nub^{2}  +3  \nuw\right)
    +\left( \nuw^{5}+2  \nuw^{2}+ \nub\right) s  + 2\left(2\nuw^4 + \nub\nuw^2 + 2\nub^2 + 3\nuw\right)^2t
    \\
    +2\left(16  \nuw^{8}+5  \nuw^{6}  \nub+10  \nub^{2}  \nuw^{4}+16  \nub^{3}  \nuw^{2}+59  \nuw^{5}+16  \nub^{4}+54  \nub  \nuw^{3}+37  \nub^{2}  \nuw+32  \nuw^{2}+11  \nub\right) ts.
\end{multline*}

Unfortunately, \eqref{eq: KP Ising aperiodic} is \emm not, a PDE for the series  $\vec{I}:= 6 D_t \widetilde I$,
because some of the terms it contains do not involve any $t$-derivative. To remedy this, we will combine the first few $t$-derivatives of~\eqref{eq: KP Ising aperiodic} to eliminate the terms with no $t$-derivative. (Note
that as soon as we obtain a PDE for $\partial_t \widetilde I$ we are done, since we can replace $\partial_t \widetilde I$ by $\vec I/(6t)$ to obtain a PDE for~$\vec I$.)
We thus need to examine the action of $\partial_t$ on the operators involved in~\eqref{eq: KP Ising aperiodic}. We say that a linear operator $\mathcal O$ is \emm almost commuting, (or AC) if there exists another linear operator $\widehat{\mathcal O}$ such that $\partial_t \circ \mathcal O= \widehat{\mathcal O} \circ \partial_t$. For instance, the operators $\partial_\bullet$ and $\partial_\circ$ are AC, as well as any operator of the form $\mathcal O \circ \partial_t$.
The multiplication by a polynomial \emm not involving~$t$, is also AC. The composition of two AC operators is still AC. These properties imply in particular that the above operators  $\widetilde{\Lambda}$,  $\widetilde{\OpLeft}$, $\widetilde{\Lambda}_2$ and $\widetilde{\Lambda}_4$ are AC. For instance,
\[
\partial_t\circ\widetilde{\Lambda}=\widehat{\Lambda}\circ \partial_t
\]
with
\[
\widehat{\Lambda}=2\left(3(\nuw^2+\nub)(D_t+1)+ \nuw(1-\nub\nuw)\partial_\bullet+(1-\nub\nuw)\partial_\circ\right).
\]
Thus, several terms in the first derivative of~\eqref{eq: KP Ising aperiodic} can already be written as differential operators acting on $\partial_t \widetilde I$.
There are two sources of terms that cannot be written in this way:
\begin{itemize}
\item multiplication by $t$ is not AC, so that problems arise from the term $t^2\widetilde \Lambda_4 \widetilde I$ for instance,
\item the quadratic term $t^2(\widetilde{\Lambda}_2 \widetilde{I})^2$ would be a problem even without its factor $t^2$, because $\partial_t (F^2)=2 F \partial_t F$ still involves the non-differentiated function $F$.
\end{itemize}
The first problem will be solved by taking higher order derivatives: for instance, $\partial_t^3 (t^2 F)$ does not contain non-differentiated terms.
Likewise, with high order derivatives the quadratic term yields terms that are only linear in $\widetilde{\Lambda}_2 \widetilde{I}$ (in the same way that $\partial_t^k (F^2)$ is only linear in $F$ for $k\ge1$). We can then eliminate $\widetilde{\Lambda}_2 \widetilde{I}$ between two successive derivatives of~\eqref{eq: KP Ising aperiodic} to obtain a PDE on $\partial_t \widetilde I$.

Let us go through this strategy in more detail. Let us differentiate~\eqref{eq: KP Ising aperiodic} a first time with respect to $t$:
\begin{align}
\partial_t\widetilde{\OpLeft} \widetilde{I}&= \frac s {12}\partial_t(t^2\widetilde{\Lambda}_4\widetilde{I})     + \frac 1 2 \partial_t\left( t^2(\widetilde{\Lambda}_2 \widetilde{I})^2\right) \notag
                                            + \left( \nuw+ 4 t\left(2\nuw^4 + \nub\nuw^2 + 2\nub^2 + 3\nuw\right)\right)\widetilde{\Lambda}_2 \widetilde{I}
                                             \notag\\
&\ \  + t\left( \nuw+2t\left(2\nuw^4 + \nub\nuw^2 + 2\nub^2 + 3\nuw\right)\right)\partial_t(\widetilde{\Lambda}_2 \widetilde{I})     + \widetilde{Q}_1,
\label{eq: first deriv raw}
\end{align}
for some polynomial $\widetilde Q_1$ of degree $1$ in $t$. Since  $\widetilde{\OpLeft}$, $\widetilde{\Lambda}$, $\widetilde{\Lambda}_2$ and $\widetilde{\Lambda}_4$ are AC,~\eqref{eq: first deriv raw} can be rewritten as follows:
\[
  \mathcal{O}_1\left(\partial_t \widetilde I\right)
  =J_1  + \widetilde{Q}_1,
\]
where $\mathcal{O}_1$ is a differential operator and 
\begin{align*}
  J_1 
  &=\frac{s}{6}t\widetilde{\Lambda}_4\widetilde{I}+t\left(\widetilde{\Lambda}_2 \widetilde{I}\right)^2+ t^2\widetilde{\Lambda}_2 \widetilde{I} \cdot
  \partial_{t}\left(\widetilde{\Lambda}_2 \widetilde{I}\right)
   + \left( \nuw+ 4 t\left(2\nuw^4 + \nub\nuw^2 + 2\nub^2 + 3\nuw\right)\right)\widetilde{\Lambda}_2 \widetilde{I}
  \\
  &=   \frac{s}{6}t\widetilde{\Lambda}_4\widetilde{I}+t\left(\widetilde{\Lambda}_2 \widetilde{I}\right)^2
    + t^2 \widetilde{\Lambda}_2  \widetilde{I}\cdot\widehat{\Lambda}_2 \left(\partial_t\widetilde{I}\right)
    + \left( \nuw+ 4 t\left(2\nuw^4 + \nub\nuw^2 + 2\nub^2 + 3\nuw\right)\right)\widetilde{\Lambda}_2 \widetilde{I},
\end{align*}
where $\widehat{\Lambda}_2 $ is defined by $\partial_t \circ \widetilde\Lambda_2= \widehat \Lambda_2\circ \partial_t$.
Differentiating a second time with respect to $t$, we obtain an equation of the form
\[
  \mathcal{O}_2 \left(\partial_t \widetilde I\right)
  = J_2  + \widetilde{Q}_0,
\]
for some polynomial $\widetilde{Q}_0$ of degree $0$ in $t$, with 
\[
 J_2
  =\frac{s}{6}\widetilde{\Lambda}_4\widetilde{I}
  +\left(\widetilde{\Lambda}_2 \widetilde{I}\right)^2
  + 4 t \widetilde{\Lambda}_2 \widetilde{I} \cdot \widehat{\Lambda}_2 \left(\partial _t \widetilde{I}\right)
  + t^2 \widetilde{\Lambda}_2  \widetilde{I}\cdot\partial_t\widehat{\Lambda}_2 \left(\partial_t\widetilde{I}\right)
    +  4 \left(2\nuw^4 + \nub\nuw^2 + 2\nub^2 + 3\nuw\right)\widetilde{\Lambda}_2 \widetilde{I}.
\]
Therefore, for the third derivative with respect to $t$, we obtain an equation of the form
\beq\label{order3}
  \mathcal{O}_3\left(\partial_t \widetilde I\right)
=  \widetilde{\Lambda}_2 \widetilde{I} \cdot \widetilde{\mathcal{O}}_3\left(\partial_t \widetilde I\right) ,
\eeq
\emm i.e., the only terms that cannot be expressed in terms of $\partial_t \widetilde I$ are linear in $\widetilde{\Lambda}_2 \widetilde{I}$. Thus, the fourth derivative has a similar form:
\beq\label{order4}
  \mathcal{O}_4\left(\partial_t \widetilde I\right)
=  \widetilde{\Lambda}_2 \widetilde{I} \cdot \widetilde{\mathcal{O}}_4\left(\partial_t \widetilde I\right).
\eeq
At this point, it suffices to take a linear combination of the last two equations to eliminate $\widetilde{\Lambda}_2 \widetilde{I}$ and obtain a PDE on $\partial _t \widetilde I$, namely
\[
\mathcal{O}_3\left(\partial_t \widetilde I\right)
     \cdot  \widetilde{\mathcal{O}}_4\left(\partial_t \widetilde I\right) =
        \mathcal{O}_4\left(\partial_t \widetilde I\right)
    \cdot  \widetilde{\mathcal{O}}_3\left(\partial_t \widetilde I\right)
    .
\]
We refer the reader to the {\sc Maple} companion file for the explicit computation.

The order and degree of the PDE can be predicted without going through all the computations. Indeed, we take the fourth derivative of the PDE~\eqref{eq: KP Ising aperiodic}, which has order $4$ in $\widetilde{I}$, and then express it in terms of $\partial_t \widetilde I$, which yields an order 7. As for the degree, the initial PDE is quadratic, and so are all its derivatives. In particular, the terms $\widetilde{\mathcal{O}}_3\left(\partial_t \widetilde I\right)$ and $\widetilde{\mathcal{O}}_4\left(\partial_t \widetilde I\right)$ in~\eqref{order3} and~\eqref{order4} have degree $1$. The final elimination thus yields an equation of degree $3$.
\end{proof}

% ============================================================
\subsection{Checking  the planar case}
\label{sec:planar-check}
% ===================================================

The planar rooted Ising series  $\vec{P}:=\vec{I}_{|s=0}$ is known explicitly in rational parametric form, see Proposition~\ref{prop: rat param} below. It is thus natural to check the PDE that we have obtained above on this case.  In fact, when $s=0$ the tricky term $t^2\widetilde\Lambda_4 \widetilde I$ disappears from~\eqref{eq: KP Ising aperiodic}, which allows us to derive a PDE of order $3$ only in the planar case. This is the PDE that we actually check.

Let us first explain how we establish this PDE. We return to the proof of Proposition~\ref{prop: PDE rooted}, when $s$ is specialized to 0. We write $\widetilde P:= \left. \widetilde I\right|_{s=0}$, but otherwise re-use the notation of this proof, even if some terms have changed because $s=0$. The first derivative of~\eqref{eq: KP Ising aperiodic} can once again be written as
\beq\label{eq: first deriv Ising planar}
  \mathcal{O}_1\left(\partial_t \widetilde P\right)
  =J_1   + \widetilde{Q}_1,
\eeq
where now 
\[
  J_1 
 =  t\left(\widetilde{\Lambda}_2 \widetilde{P}\right)^2
  + t^2 \widetilde{\Lambda}_2  \widetilde{P}\cdot{\widehat{\Lambda}_2 }\left(\partial_t\widetilde{P}\right)    + \left( \nuw+ 4 t\left(2\nuw^4 + \nub\nuw^2 + 2\nub^2 + 3\nuw\right)\right)\widetilde{\Lambda}_2\widetilde{P}.
%\widetilde{I}}.
\]
Thus,~\eqref{eq: first deriv Ising planar} is already a polynomial of degree $2$ in $\widetilde{\Lambda}_2 \widetilde{P}$, whose coefficients are differential polynomials in $\partial_t\widetilde{P}$. Then, the second derivative can be written as
\beq\label{eq: second deriv Ising planar}
  \mathcal{O}_2\left(\partial_t \widetilde P\right)
  =J_2   + \widetilde{Q}_2,
  \eeq
  where
  \[
J_2
   =\left(\widetilde{\Lambda}_2 \widetilde{P}\right)^2
  + 4t \widetilde{\Lambda}_2 \widetilde{P} \cdot \widehat{\Lambda}_2 \left(\partial _t \widetilde{P}\right)
  + t^2\widetilde{\Lambda}_2  \widetilde{P}\cdot\partial_t\widehat{\Lambda}_2 \left(\partial_t\widetilde{P}\right)
     +  4 \left(2\nuw^4 + \nub\nuw^2 + 2\nub^2 + 3\nuw\right)\widetilde{\Lambda}_2 \widetilde{P}.
\]
So~\eqref{eq: second deriv Ising planar} is also a polynomial of degree 2 in $\widetilde{\Lambda}_2 \widetilde{P}$. We can now eliminate $\widetilde{\Lambda}_2 \widetilde{P}$ between~\eqref{eq: first deriv Ising planar} and~\eqref{eq: second deriv Ising planar} by computing a resultant. We thus obtain a PDE for $\partial_t \widetilde P=  \vec{P}/(6t)$,
and finally for $\vec P$ itself. It has order 3 only, but degree 4. This should be compared to the case $s=0$ of the PDE of Proposition~\ref{prop: PDE rooted}, which has order 5 and degree 3.

\medskip
Let us now give an explicit  expression of $\vec P$, in rational parametric form.

\begin{prop}
  The series $\vec{P}$ is algebraic over $\qs(t,\nub,\nuw)$, and the algebraic variety $\{(t,\nub,\nuw,\vec{P}(t,\nub,\nuw))\}$ admits the following rational parametrization by a triple of  variables $(\Ab,\Aw,G)$:
\begin{align*}
  \nuw&=-\frac{\Aw^2-2G(\Ab+\Aw^2)+8\Ab G^2-\Ab^2\Aw}{4G^2-\Ab\Aw}, \qquad
                \nub=-\frac{\Ab^2-2G(\Ab^2+\Aw)+8\Aw G^2-\Ab\Aw^2}{4G^2-\Ab\Aw},\\
  t(1-\nub\nuw)^3&=:T=-\frac{G}{(4G^2-\Ab\Aw)^2}\cdot\Big(32G^5 - 16G^4 - 16G^3\Ab \Aw
                   + 4(\Ab^2\Aw^2 + \Ab^3 + {\color{red}4}\Aw^3 + 3\Ab\Aw)G^2 \\
&\ \ - 2(\Ab^4\Aw + \Ab\Aw^4 + 3\Ab^2\Aw^2 + \Ab^3 + \Aw^3)G + \Ab\Aw(\Ab^2\Aw^2 + \Ab^3 + \Aw^3)\Big),
\end{align*} 
and
\[
\vec{P}=\frac W {T^2},
\]
where
\begin{align*}     
W&
= \frac{G}{(\Ab \Aw - 4  G^2)^3}\cdot\Big(384 G^9 - 128 G^8 + 32(16\Ab^3 + 16\Aw^3 + 15\Ab\Aw) G^7  \\
-&32(3\Ab^2\Aw^2 + 5\Ab^3 + 5\Aw^3) G^6 + 8(8\Ab^6 + 8\Aw^6 - 30\Ab^4\Aw - 30\Ab\Aw^4 - 39\Ab^2\Aw^2 + 2\Ab^3 + 2\Aw^3) G^5 \\
-& 8(4\Ab^5\Aw^2 + 4\Ab^2\Aw^5 + 4\Ab^6 - 4\Ab^3\Aw^3 + 4\Aw^6 - 14\Ab^4\Aw - 14\Ab\Aw^4 - 9\Ab^2\Aw^2) G^4 \\
-& 2(16\Ab^7\Aw + 8\Ab^4\Aw^4 + 16\Ab\Aw^7 + 10\Ab^5\Aw^2 + 10\Ab^2\Aw^5 + 4\Ab^6 + 35\Ab^3\Aw^3 + 4\Aw^6 + 18\Ab^4\Aw + 18\Ab\Aw^4) G^3 \\
+ &2(14\Ab^6\Aw^3 + 14\Ab^3\Aw^6 + 12\Ab^7\Aw + 39\Ab^4\Aw^4 + 12\Ab\Aw^7 + 15\Ab^5\Aw^2 + 15\Ab^2\Aw^5 + 2\Ab^6 + 4\Ab^3\Aw^3 + 2\Aw^6) G^2 \\
+& 2\Aw\Ab(\Ab^7\Aw  + \Ab\Aw^7 - 4\Ab^4\Aw^4- 6\Ab^5\Aw^2 - 6\Ab^2\Aw^5 - 2\Ab^6 - 4\Ab^3\Aw^3 - 2\Aw^6) G  \\
-&\Ab^2\Aw^2(\Ab + \Aw)(\Ab^2 - \Ab\Aw + \Aw^2)(\Ab^2\Aw^2 + \Ab^3 + \Aw^3)\Big).
\end{align*}
\label{prop: rat param}
\end{prop}

We prove this proposition in Appendix~\ref{sec: appendix parametrization}, starting from Theorem~8.3.1 in~\cite{eynard-book}. This theorem  deals with slightly different Ising maps (in particular, maps having a boundary), and the enumeration variables also differ from ours, so it takes a bit of work to obtain the above proposition. Similar rational parametrizations have been given in the literature~\cite{BK87,mbm-schaeffer-ising}, but we could not find any ready-to-use statement. 

Let us now return to the third order PDE obtained above for $\vec P$. We perform the change of variables $(\nub, \nuw, t) \mapsto (\Ab, \Aw, G)$ of Proposition~\ref{prop: rat param} (as we did in Lemma~\ref{lem_change} to go from $p_2, q_2,z$ to $\nub, \nuw, t$), replace $\vec P$ by its rational expression in terms of $\Ab, \Aw$ and $G$, and check in {\sc Maple} that the resulting rational expression is indeed zero.

%%%%%%%%%%%%%%%%%%%%%%%%%%%%%%%%%%%%%%%%%%%%%%%%%%%%%%%%%%%%%%%%%%%%%%%% 
\appendix

%%%%%%%%%%%%%%%%%%%%%%%%%%%%%%%%%%%%%%%%%%%%%%%%%%%%%%%%%%%%%%%%%%%%%%%% 
\section{Rational parametrization for  planar rooted     maps}
\label{sec: appendix parametrization}
%%%%%%%%%%%%%%%%%%%%%%%%%%%%%%%%%%%%%%%%%%%%%%%%%%%%%%%%%%%%%%%%%%%%%%%% 

Our aim here is to derive the rational parametrization of Proposition~\ref{prop: rat param} from earlier results. We start from Theorem~~8.3.1 in~\cite{eynard-book}, which deals with an Ising model dual to the one considered in this paper: \emm triangulations, with spins on \emm faces,. Converted to our setting, this theorem gives a rational parametrization of the \gf \ $W^{(0)}_1(x)$  counting  vertex bicolored maps, rooted at a black vertex of arbitrary degree, while all other vertices have degree $3$.
More precisely:
\[
  W^{(0)}_1(x):=\sum_\m \frac{T^{\ff(\m)} c_{\bullet}^{\ee^{\bullet}(\m)}c_{\circ}^{\ee^{\circ}(\m)}c_{\bullet\circ}^{\ee^{\bullet\circ}(\m)}}{x^{d(\m)+1}},
\]
where the functions $\ee^\bullet$, $\ee^\circ$ and $\ee^{\bullet\circ}$ have been defined at the beginning of Section~\ref{sec:Ising},  $\ff$ is the number of faces and $d$  the degree of the root vertex.

Several steps are needed to obtain $\vec{P}$ from Theorem 8.3.1 in~\cite{eynard-book}. We refer to the {\sc Maple}  companion file for computation details. In particular:
\begin{itemize}
   \item The series  $W^{(0)}_1(x)$  keeps track of faces  rather than edges.
\item It  counts Ising \emm near-cubic, maps (with a root vertex of arbitrary degree),   so that the \gf \ $W^{(0)}$ counting true cubic maps
   is  the coefficient of $x^{-4}$ in $W^{(0)}_1(x)$. We will extract it using  Lagrange inversion.
\item The cubic maps  counted by  $W^{(0)}$  are rooted at a black vertex,
   so that we need to symmetrize by considering $W^{(0)}+W^{(0)}_
  { | c_{\circ} \leftrightarrow c_{\bullet}}  $.
 \end{itemize}
 
Let us now give some details. Theorem 8.3.1 in~\cite{eynard-book} gives a parametrization of
\beq\label{Yx-def}
Y(x):=\frac{1}{c}\left({x^2}- W^{(0)}_1(x)\right)
\eeq 
where $c_{\bullet}$, $c_{\circ}$, $c_{\bullet\circ}$ are first rewritten in terms of new variables  $a$, $b$, $c$ as follows: 
\[
c_{\bullet}= \frac{b}{ab-c^2}, \qquad
 c_{\circ}= \frac{a}{ab-c^2}, \qquad 
c_{\bullet\circ}= \frac{c}{ab-c^2}.
\]
The parameter $c$ is redundant and we take $c=-1$.
Next  we need to perform a change of variables from $(T,a,b,W^{(0)})$ to $ (t,\nub, \nuw,\vec{P})$. Since the contribution  in $W^{(0)}$ of a cubic map $\m$ having $3n$ edges is 
\begin{align*}
T^{\ff }c_{\bullet}^{\ee^{\bullet} }c_{\circ}^{\ee^{\circ} }c_{\bullet\circ}^{\ee^{\bullet\circ} }&=T^{n+2}\left(\frac{b}{ab-1}\right)^{\ee^{\bullet} }\left(\frac{a}{ab-1}\right)^{\ee^{\circ} }\left(\frac{-1}{ab-1}\right)^{3n-(\ee^{\bullet} +\ee^{\circ} )}\\
&=T^2\left(\frac{-T}{(ab-1)^3}\right)^n\left(-b\right)^{\ee^{\bullet} }\left(-a\right)^{\ee^{\circ} }
\end{align*}
(where we write $\ff$ rather than $\ff(\m)$ to lighten notation, and similarly for edge numbers), the appropriate change of variables is
\beq   
t= \frac{-T}{(ab-1)^3},\qquad 
\nub=-b, \qquad
\nuw= -a.
\label{eq: change-var-rooted-param}
\eeq                
Then the series $\vec P$ of Section~\ref{sec:planar-check} is finally
\beq\label{vecP-W}
\vec{P}=\frac{W^{(0)}}{T^2}+\left(\frac{W^{(0)}}{T^2}\right)_{| a \leftrightarrow b}.
\eeq

Let us now go back to Eynard's parametrization of $W_1^{(0)}(x)$, recalling the definition~\eqref{Yx-def} of the related series $Y(x)$. It gives $x$ and $Y(x)$ as Laurent polynomials in a parameter $z$:
\beq
\label{xy-param}
\begin{cases}
x(z)&=\gamma z +\sum_{k=0}^2\alpha_kz^{-k},\\
Y(x(z)):=y(z)&= \gamma z^{-1} +\sum_{k=0}^2\beta_kz^{k},
\end{cases}
\eeq
where the $7$
% additional
parameters $\gamma, \alpha_i$ and $\beta_i$  are coupled to $a, b$ and $T$ by the following relations:
\[
\begin{cases}
  ax(z)-x(z)^2+y(z)&\underset{z\to\infty}{\sim} \frac{T}{\gamma z}
  +\LandauO(z^{-2}),\\
by(z)-y(z)^2+x(z)&\underset{z\to 0}{\sim} \frac{Tz}{\gamma}+\LandauO(z^2).
\end{cases} 
\]
These relations allow us to  express the original variables $a, b$ and $T$, but also the $4$ parameters  $\alpha_1, \alpha_2, \beta_1, \beta_2$, as  rational functions in the  $3$ main parameters $\gamma$, $\alpha_0$, $\beta_0$. In particular,  $\alpha_2=\beta_2=\gamma^2$ and
\begin{align}
a&=\frac{\beta_0^2-2\gamma^2(\alpha_0+\beta_0^2)+8\alpha_0\gamma^4-\alpha_0^2\beta_0}{4\gamma^4-\alpha_0\beta_0}, \nonumber \\
b&=\frac{\alpha_0^2-2\gamma^2(\alpha_0^2+\beta_0)+8\beta_0\gamma^4-\alpha_0\beta_0^2}{4\gamma^4-\alpha_0\beta_0}, \label{abT-param}\\
  T&=-\frac{\gamma^2}{(4\gamma^4-\alpha_0\beta_0)^2}\cdot
     \Big(32\gamma^{{10}} - 16\gamma^{ 8} - 16\gamma^{ 6}\alpha_0 \beta_0
     +4 \left(\alpha_{0}^{2} \beta_{0}^{2}+\alpha_{0}^{3}+\beta_{0}^{3}+3 \alpha_{0} \beta_{0}\right)
     \gamma^{ 4} \nonumber\\
  &\ -2 \left(\alpha_{0}^{4} \beta_{0}+\alpha_{0} \beta_{0}^{4}+3 \alpha_{0}^{2} \beta_{0}^{2}+\alpha_{0}^{3}+\beta_{0}^{3}\right)
    \gamma^{ 2} + \alpha_0\beta_0(\alpha_0^2\beta_0^2 + \alpha_0^3 + \beta_0^3)\Big).\nonumber
\end{align}
We refer to our {\sc Maple} session for expressions of $\alpha_1$ and $\beta_1$.  Combined with~\eqref{eq: change-var-rooted-param}, the above three equations  give the expressions of $\nub, \nuw$ and $t$ stated in Proposition~\ref{prop: rat param}   (in this proposition we write $\Ab:=\alpha_0$, $\Aw:=\beta_0$, and $G:=\gamma^2$ to fit better with the notational conventions of the paper).

 Next, to obtain a parametrization for $W^{(0)}$, we can apply the Lagrange inversion formula to the system~\eqref{xy-param}. Indeed, denoting $\bx:=x^{-1}$ and $\bz:=z^{-1}$, we have:
\[
  \begin{cases}
    \bz &= \bx \left(\gamma + \alpha_0\bz + \alpha_1\bz^2 + \alpha_2\bz^3\right)
  :=\bx \, \Phi(\bz),\\
  y(z) & \displaystyle = \gamma\bz+\beta_0+\frac{\beta_1}{\bz}+\frac{\beta_2}{\bz^{2}}:= \Upsilon(\bz),
\end{cases}
\]
so that the coefficient of $\bx^4$ in $Y(x)$ is
\[
  W^{(0)}(\gamma,\alpha_0,\beta_0)
  =\frac{1}{4}[u^3]\left(\Upsilon'(u)\Phi(u)^4\right).
\]
With the help of {\sc Maple} we obtain
\begin{align*}
  W^{(0)}&=\frac{\gamma^2}{(\alpha_0\beta_0-4\gamma^4)^3}\cdot
           \left(
           192 \gamma^{18}-64 \gamma^{16}+16 \alpha_{0} \left(32 \alpha_{0}^{2}+15 \beta_{0}\right) \gamma^{14}-16 \left(3 \alpha_{0}^{2} \beta_{0}^{2}+13 \alpha_{0}^{3}-3 \beta_{0}^{3}\right) \gamma^{12}
           \right.
  \\
         &\   +4 \left(16 \alpha_{0}^{6}-54 \alpha_{0}^{4} \beta_{0}-6 \alpha_{0} \beta_{0}^{4}-39 \alpha_{0}^{2} \beta_{0}^{2}+6 \alpha_{0}^{3}-2 \beta_{0}^{3}\right) \gamma^{10}
  \\
         &\   -4 \alpha_{0}^{2} \left(8 \alpha_{0}^{3} \beta_{0}^{2}+8 \alpha_{0}^{4}-4 \alpha_{0} \beta_{0}^{3}-28 \alpha_{0}^{2} \beta_{0}-9 \beta_{0}^{2}\right) \gamma^{8}
  \\
 &\  -\alpha_{0} \left(32 \alpha_{0}^{6} \beta_{0}+8 \alpha_{0}^{3} \beta_{0}^{4}+14 \alpha_{0}^{4} \beta_{0}^{2}+6 \alpha_{0} \beta_{0}^{5}+8 \alpha_{0}^{5}+35 \alpha_{0}^{2} \beta_{0}^{3}+30 \alpha_{0}^{3} \beta_{0}+6 \beta_{0}^{4}\right) \gamma^{6}\\
&\   +\alpha_{0}^{2} \left(24 \alpha_{0}^{4} \beta_{0}^{3}+4 \alpha_{0} \beta_{0}^{6}+24 \alpha_{0}^{5} \beta_{0}+39 \alpha_{0}^{2} \beta_{0}^{4}+21 \alpha_{0}^{3} \beta_{0}^{2}+9 \beta_{0}^{5}+4 \alpha_{0}^{4}+4 \alpha_{0} \beta_{0}^{3}\right) \gamma^{4}\\
         &\   +\alpha_{0}^{3} \beta_{0} \left(2 \alpha_{0}^{5} \beta_{0}-4 \alpha_{0}^{2} \beta_{0}^{4}-9 \alpha_{0}^{3} \beta_{0}^{2}-3 \beta_{0}^{5}-4 \alpha_{0}^{4}-4 \alpha_{0} \beta_{0}^{3}\right) \gamma^{2}
-\alpha_0^5\beta_0^2(\alpha_0^2\beta_0^2 + \alpha_0^3 + \beta_0^3)
\Big).
\end{align*}
Observe on the system~\eqref{abT-param} that the exchange of $\alpha_0$ and $\beta_0$ exchanges $a$ and $b$ and leaves~$T$ unchanged. In other words, it swaps the two colors $\circ$ and $\bullet$. Thus, \eqref{vecP-W} rewrites as:
\[
\vec{P}=\frac{W^{(0)}(\gamma,\alpha_0,\beta_0)+W^{(0)}(\gamma,\beta_0,\alpha_0)}{T(\gamma,\alpha_0,\beta_0)^2}=:\frac{W(\gamma,\alpha_0,\beta_0)}{T(\gamma,\alpha_0,\beta_0)^2},
\]
where $W$ is given in Proposition~\ref{prop: rat param}.

\medskip
\noindent {\bf Acknowledgments.} We warmly thank Marc Mezzarobba for helpful discussions regarding the implementation of the algorithm of Section~\ref{sec:implement} in \textsc{SageMath}.

%%%%%%%%%%%%%%%%%%%%%%%%%%%%%%%%%%%%%%%%%%%%%%%%%%%%%%%%
\bibliographystyle{abbrv}
\bibliography{Biblio}

\begin{thebibliography}{10}

\bibitem{Bender-Canfield-orientable}
E.~A. Bender and E.~R. Canfield.
\newblock The number of rooted maps on an orientable surface.
\newblock {\em J. Combin. Theory Ser. B}, 53(2):293--299, 1991.
\newblock \href{https://doi.org/10.1016/0095-8956(91)90079-Y}{doi}.

\bibitem{BLL-book}
F.~Bergeron, G.~Labelle, and P.~Leroux.
\newblock {\em Combinatorial species and tree-like structures}, volume~67 of
  {\em Encyclopedia of Mathematics and its Applications}.
\newblock Cambridge University Press, Cambridge, 1998.

\bibitem{Bershadsky}
M.~A. Bershadsky and A.~A. Migdal.
\newblock Ising model of a randomly triangulated random surface as a definition
  of fermionic string theory.
\newblock {\em Phys. Lett. B}, 174:393--398, 1986.

\bibitem{BCD-non-oriented}
V.~Bonzom, G.~Chapuy, and M.~Do{łę}ga.
\newblock Enumeration of non-oriented maps via integrability.
\newblock {\em Algebr. Comb.}, 5(6):1363--1390, 2022.
\newblock \href{https://arxiv.org/abs/2110.12834}{arXiv:2110.12834},
  \href{https://doi.org/10.5802/alco.268}{doi}.

\bibitem{BK87}
D.~V. Boulatov and V.~A. Kazakov.
\newblock The {I}sing model on a random planar lattice: the structure of the
  phase transition and the exact critical exponents.
\newblock {\em Phys. Lett. B}, 186(3-4):379--384, 1987.

\bibitem{mbm-schaeffer-ising}
M.~Bousquet-M{\'e}lou and G.~Schaeffer.
\newblock The degree distribution of bipartite planar maps: applications to the
  {I}sing model.
\newblock In K.~Eriksson and S.~Linusson, editors, {\em Formal Power Series and
  Algebraic Combinatorics}, pages 312--323, Vadstena, Sweden, 2003.
\newblock Long version on
  \href{http://arxiv.org/abs/math/0211070}{arXiv:math/0211070}.

\bibitem{BudzLouf}
T.~Budzinski and B.~Louf.
\newblock Local limits of uniform triangulations in high genus.
\newblock {\em Invent. Math.}, 223(1):1--47, 2021.
\newblock \href{https://arxiv.org/abs/2012.05813}{arXiv:2012.05813},
  \href{https://doi.org/10.1007/s00222-020-00986-3}{doi}.

\bibitem{CarrellChapuy}
S.~R. Carrell and G.~Chapuy.
\newblock Simple recurrence formulas to count maps on orientable surfaces.
\newblock {\em J. Comb. Theory, Ser. A}, 133:58--75, 2015.
\newblock \href{https://arxiv.org/abs/1402.6300}{arXiv:1402.6300},
  \href{https://doi.org/10.1016/j.jcta.2015.01.005}{doi}.

\bibitem{chapuy-structure-unicellular}
G.~Chapuy.
\newblock The structure of unicellular maps, and a connection between maps of
  positive genus and planar labelled trees.
\newblock {\em Probab. Theory Related Fields}, 147(3-4):415--447, 2010.
\newblock \href{https://arxiv.org/abs/0804.0546}{arXiv:0804.0546},
  \href{https://doi.org/10.1007/s00440-009-0211-0}{doi}.

\bibitem{chapuy-new}
G.~Chapuy.
\newblock A new combinatorial identity for unicellular maps, via a direct
  bijective approach.
\newblock {\em Adv. in Appl. Math.}, 47(4):874--893, 2011.
\newblock \href{https://arxiv.org/abs/1006.5053}{arXiv:1006.5053},
  \href{https://doi.org/10.1016/j.aam.2011.04.004}{doi}.

\bibitem{CFF13}
G.~Chapuy, V.~F\'eray, and E.~Fusy.
\newblock A simple model of trees for unicellular maps.
\newblock {\em J. Combin. Theory Ser. A}, 120(8):2064--2092, 2013.
\newblock \href{https://arxiv.org/abs/1202.3252}{arXiv:1202.3252},
  \href{https://doi.org/10.1016/j.jcta.2013.08.003}{doi}.

\bibitem{DFGZJ}
P.~Di~Francesco, P.~Ginsparg, and J.~Zinn-Justin.
\newblock {$2$}{D} {G}ravity and random matrices.
\newblock {\em Phys. Rep.}, 254(1-2):1--133, 1995.
\newblock \href{https://arxiv.org/abs/hep-th/9306153}{arXiv:hep-th/9306153},
  \href{https://doi.org/10.1016/0370-1573%2894%2900084-G}{doi}.

\bibitem{DYZ}
B.~Dubrovin, D.~Yang, and D.~Zagier.
\newblock Classical {Hurwitz} numbers and related combinatorics.
\newblock {\em Mosc. Math. J.}, 17(4):601--633, 2017.

\bibitem{eynard2matrix}
B.~Eynard.
\newblock Large-{$N$} expansion of the 2-matrix model.
\newblock {\em J. High Energy Phys.}, 2003(1):051, 38 pp., 2003.
\newblock \href{https://arxiv.org/abs/hep-th/0210047}{arXiv:hep-th/0210047},
  \href{https://doi.org/10.1088/1126-6708/2003/01/051}{doi}.

\bibitem{eynard-book}
B.~Eynard.
\newblock {\em Counting surfaces}, volume~70 of {\em Progress in Mathematical
  Physics}.
\newblock Birkh\"auser/Springer, [Cham], 2016.
\newblock CRM Aisenstadt chair lectures.

\bibitem{goulden-jackson-KP}
I.~P. Goulden and D.~M. Jackson.
\newblock The {KP} hierarchy, branched covers, and triangulations.
\newblock {\em Adv. Math.}, 219(3):932--951, 2008.
\newblock \href{https://arxiv.org/abs/0803.3980}{arXiv:0803.3980},
  \href{http://dx.doi.org/10.1016/j.aim.2008.06.013}{doi}.

\bibitem{goupil-schaeffer}
A.~Goupil and G.~Schaeffer.
\newblock Factoring {$n$}-cycles and counting maps of given genus.
\newblock {\em European J. Combin.}, 19(7):819--834, 1998.
\newblock \href{https://doi.org/10.1006/eujc.1998.0215}{doi}.

\bibitem{Ising}
E.~Ising.
\newblock Beitrag zur theorie des ferromagnetismus.
\newblock {\em Z. Phys.}, 31, 1925.

\bibitem{Ka86}
V.~A. Kazakov.
\newblock Ising model on a dynamical planar random lattice: exact solution.
\newblock {\em Phys. Lett. A}, 119(3):140--144, 1986.

\bibitem{KKN}
V.~A. Kazakov, I.~K. Kostov, and N.~Nekrasov.
\newblock D-particles, matrix integrals and {KP} hierarchy.
\newblock {\em Nucl. Phys., B}, 557(3):413--442, 1999.
\newblock \href{https://arxiv.org/abs/hep-th/9810035}{arXiv:hep-th/9810035},
  \href{https://doi.org/10.1016/S0550-3213(99)00393-4}{doi}.

\bibitem{Kazarian-Zograf}
M.~Kazarian and P.~Zograf.
\newblock Virasoro constraints and topological recursion for {G}rothendieck's
  dessin counting.
\newblock {\em Lett. Math. Phys.}, 105(8):1057--1084, 2015.
\newblock \href{https://arxiv.org/abs/1406.5976}{arXiv:1406.5976},
  \href{https://doi.org/10.1007/s11005-015-0771-0}{doi}.

\bibitem{Kontsevich}
M.~Kontsevich.
\newblock Intersection theory on the moduli space of curves and the matrix
  {Airy} function.
\newblock {\em Commun. Math. Phys.}, 147(1):1--23, 1992.
\newblock \href{http://projecteuclid.org/euclid.cmp/1104250524}{doi}.

\bibitem{Lenz}
W.~Lenz.
\newblock {Beitrag zum Verständnis der magnetischen Erscheinungen in festen
  Körpern}.
\newblock {\em Z. Phys.}, 21:613--615, 1920.

\bibitem{Louf-bijection}
B.~Louf.
\newblock A new family of bijections for planar maps.
\newblock {\em J. Combin. Theory Ser. A}, 168:374--395, 2019.
\newblock \href{https://arxiv.org/abs/1806.02362}{arXiv:1806.02362},
  \href{https://doi.org/10.1016/j.jcta.2019.06.006}{doi}.

\bibitem{LoufFormula}
B.~Louf.
\newblock Simple formulas for constellations and bipartite maps with prescribed
  degrees.
\newblock {\em Can. J. Math.}, 73(1):160--176, 2021.
\newblock \href{https://arxiv.org/abs/1904.05371}{arXiv:1904.05371}.

\bibitem{MiwaJimboDate}
T.~Miwa, M.~Jimbo, and E.~Date.
\newblock {\em Solitons: differential equations, symmetries and infinite
  dimensional algebras. {Transl}. from the {Japanese} by {Miles} {Reid}},
  volume 135 of {\em Camb. Tracts Math.}
\newblock Cambridge: Cambridge University Press, reprint of the 2000 hardback
  edition edition, 2012.

\bibitem{mullin-nemeth-schellenberg}
R.~C. Mullin, E.~Nemeth, and P.~J. Schellenberg.
\newblock The enumeration of almost cubic maps.
\newblock In {\em Proc. {L}ouisiana {C}onf. on {C}ombinatorics, {G}raph
  {T}heory and {C}omputing}, pages 281--295. Louisiana State Univ., Baton
  Rouge, La., 1970.

\bibitem{Okounkov}
A.~Okounkov.
\newblock Toda equations for {Hurwitz} numbers.
\newblock {\em Math. Res. Lett.}, 7(4):447--453, 2000.
\newblock \href{https://arxiv.org/abs/math/0004128}{arXiv:math/0004128},
  \href{https://doi.org/10.4310/MRL.2000.v7.n4.a10}{doi}.

\bibitem{PiSaSo}
C.~Pivoteau, B.~Salvy, and M.~Soria.
\newblock Algorithms for combinatorial structures: well-founded systems and
  {N}ewton iterations.
\newblock {\em J. Combin. Theory Ser. A}, 119(8):1711--1773, 2012.
\newblock \href{https://arxiv.org/abs/1109.2688}{arXiv:1109.2688},
  \href{https://doi.org/10.1016/j.jcta.2012.05.007}{doi}.

\bibitem{schabanel-bij}
J.~Schabanel.
\newblock Slit-slide-sew bijections for planar bipartite maps with prescribed
  degree.
\newblock \href{https://arxiv.org/abs/2502.07582}{arXiv:2502.07582}, 2025.

\bibitem{tutte-triangulations}
W.~T. Tutte.
\newblock A census of planar triangulations.
\newblock {\em Canad. J. Math.}, 14:21--38, 1962.
\newblock \href{https://doi.org/10.4153/CJM-1962-002-9}{doi}.

\bibitem{tutte-census-maps}
W.~T. Tutte.
\newblock A census of planar maps.
\newblock {\em Canad. J. Math.}, 15:249--271, 1963.
\newblock \href{https://doi.org/10.4153/CJM-1963-029-x}{doi}.

\bibitem{tutte-general}
W.~T. Tutte.
\newblock On the enumeration of planar maps.
\newblock {\em Bull. Amer. Math. Soc.}, 74:64--74, 1968.

\bibitem{walsh-lehman-I}
T.~R.~S. Walsh and A.~B. Lehman.
\newblock Counting rooted maps by genus. {I}.
\newblock {\em J. Combinatorial Theory Ser. B}, 13:192--218, 1972.
\newblock \href{https://doi.org/10.1016/0095-8956(72)90056-1}{doi}.

\bibitem{Witten}
E.~Witten.
\newblock Two-dimensional gravity and intersection theory on moduli space.
\newblock In {\em Surveys in differential geometry. Vol. I: Proceedings of the
  conference on geometry and topology, held at Harvard University, Cambridge,
  MA, USA, April 27-29, 1990}, pages 243--310. Providence, RI: American
  Mathematical Society; Bethlehem, PA: Lehigh University, 1991.

\end{thebibliography}

\end{document}